\documentclass[12pt]{amsart}
\usepackage{amsfonts,amscd,latexsym,amsmath,amssymb,enumerate,mathrsfs,hyperref}
\theoremstyle{plain}
\newtheorem{master}{Master}[section]
\newtheorem{prop}[master]{Proposition}
\newtheorem{thm}[master]{Theorem}
\newtheorem{fact}[master]{Fact}
\newtheorem{lem}[master]{Lemma}
\newtheorem{cor}[master]{Corollary}
\newtheorem{question}[master]{Question}
\newtheorem{claim}[master]{Claim}
\theoremstyle{definition}
\newtheorem{defin}[master]{Definition}
\newtheorem{observation}[master]{Observation}
\theoremstyle{remark}
\newtheorem{remark}[master]{Remark}
\numberwithin{equation}{section}

\newcommand{\Ur}{\mathbb{U}}
\newcommand{\Rea}{\mathbb{R}}
\newcommand{\Nat}{\mathbb{N}}
\newcommand{\Int}{\mathbb{Z}}
\newcommand{\Rat}{\mathbb{Q}}
\newcommand{\Age}{\mathcal{K}}

\newcommand{\Gr}{\mathcal{G}}
\newcommand{\dist}{\mathrm{dist}}
\newcommand{\intdist}{\mathrm{int\, dist}}
\begin{document}
\title{Metrical universality for groups}
\author[M. Doucha]{Michal Doucha}
\address{Institute of Mathematics\\ Polish Academy of Sciences\\
00-656 Warszawa, Poland}
\address{Laboratoire de Math\' ematiques de Besan\c con\\Universit\' e de Franche-Comt\' e\\France} 
\email{m.doucha@post.cz}
\thanks{The author was supported by funds allocated to the implementation of the international co-funded project in the years 2014-2018, 3038/7.PR/2014/2, and by the EU grant PCOFUND-GA-2012-600415.}
\begin{abstract}
We prove that for any constant $K>0$ there exists a separable group equipped with a complete bi-invariant metric bounded by $K$, isometric to the Urysohn sphere of diameter $K$, that is of `almost-universal disposition'. It is thus an object in the category of separable groups with bi-invariant metric analogous in its properties to the Gurarij space from the category of separable Banach spaces. We show that this group contains an isometric copy of any separable group equipped with bi-invariant metric bounded by $K$. As a consequence, we get that it is a universal Polish group admitting compatible bi-invariant metric, resp. universal second countable SIN group. Moreover, the almost-universal disposition shows that the automorphism group of this group is rich and it characterizes the group uniquely up to isometric isomorphism. We also show that this group is in a certain sense generic in the class of separable group with bi-invariant metric (bounded by $K$).

On the other hand, we prove there is no metrically universal separable group with bi-invariant metric when there is no restriction on diameter. The same is true for separable locally compact groups with bi-invariant metric.

Assuming the generalized continuum hypothesis, we prove that there exists a metrically universal (unbounded) group of density $\kappa$ with bi-invariant metric for any uncountable cardinal $\kappa$. We moreover deduce that under GCH there is a universal SIN group of weight $\kappa$ for any infinite cardinal $\kappa$.
\end{abstract}
\maketitle
\section*{Introduction}
There has been an effort in trying to find various Polish groups universal for certain classes of Polish groups usually considering two different notions of universality: injective and surjective universality. If $\Gr$ is some class of Polish groups (let us say closed under isomorphism and taking closed subgroups) then $G\in \Gr$ is universal for $\Gr$ (or injectively universal) if every $H\in \Gr$ is topologically isomorphic with some closed subgroup of $G$. On the other hand, we say that $G\in \Gr$ is projectively universal if for every $H\in \Gr$ there is a continuous surjective homomorphism from $G$ onto $H$; or equivalently, there is a closed normal subgroup $N_H\leq G$ such that $H$ is topologically isomorphic to the quotient group $G/N_H$.

Various classes were considered in the past. Schreier and Ulam proved that there are both injectively and projectively universal compact Polish groups (see \cite{SB}, for example, comment starting on p. 184). On the other hand, there is no projectively and injectively universal locally compact Polish group (see again \cite{SB}). In \cite{Us1}, resp. \cite{Us2}, Uspenskij proved that the group of all homeomorphisms of the Hilbert cube ($\mathrm{Homeo}([0,1]^\omega)$) with the compact-open topology, resp. the group of all isometries of the Urysohn space with the point-wise convergence (or equivalently the compact-open) topology, are the (injectively) universal Polish groups (i.e. universal for the class of all Polish groups). Recently in \cite{BY}, Ben-Yaacov proved that the group of all linear isometries of the Gurarij space is a universal Polish group as well. Moreover, it is not hard to derive from the Uspenskij's result that the group of all linear isometries of the Holmes' space, the Lipschitz-free Banach 
space over the Urysohn space, is also a universal Polish group (we refer to Chapter 5 in \cite{Pe} for information on this Banach space).

In \cite{SPW}, Shakhmatov, Pelant and Watson constructed, based on a work of Graev in \cite{Gr}, a projectively universal abelian Polish group and projectively universal Polish group admitting bi-invariant metric. Shkarin in \cite{Sk} found a (injectively) universal abelian group (further investigated by Niemiec in \cite{Nie1}). Another natural class of Polish groups are those that admit a compatible complete left-invariant metric (recall that every Polish group admits a compatible complete metric and a compatible left-invariant metric, however these metrics might differ as is the case for example with $S_\infty$, the permutation group of integers, see \cite{Gao} for example). They are usually called CLI (complete left-invariant) groups. Malicki in \cite{Ma} proved that there is neither a universal CLI group nor a projectively universal CLI group. The last major contribution to this area were in \cite{Ding} the Ding's construction of a projectively universal Polish group (i.e. projectively universal for the 
whole class of Polish groups) and then further constructions of projectively universal Polish groups by Pestov and Uspenskij in \cite{PeUs}. They in fact construct projectively universal countable metrizable groups whose completions are then projectively universal Polish groups.

It was left open whether there is a universal Polish group in the class of Polish groups admitting bi-invariant metric. It is well known and not difficult to check that the condition admitting a compatible bi-invariant metric is equivalent with having a countable neighborhood base of the identity where each of these neighborhoods is conjugation-invariant. These groups are sometimes called (metrizable) SIN groups (small-invariant-neighborhoods). Thus, stated somewhat more generally than just for Polish groups, the consequence of our main result is that there is a universal second-countable SIN group.

For the reader's convenience we summarize the universality results for Polish groups in the following table.\\

\begin{tabular}[t]{|c|c|c|}
\hline
\bf{Class of groups} & \bf{Universal object} & \bf{Surj. univ. object}\\ \hline\hline
all Polish groups & yes \cite{Us1},\cite{Us2},\cite{BY} & yes \cite{Ding},\cite{PeUs}\\ \hline
abelian Polish groups & yes \cite{Sk} & yes \cite{SPW}\\ \hline
compact Polish groups & yes, \cite{SB} & yes,\cite{SB}\\ \hline
loc. compact P.groups & no, \cite{SB} & no, \cite{SB}\\ \hline
CLI groups & no \cite{Ma} & no \cite{Ma}\\ \hline
SIN P. groups & yes, this paper & yes, \cite{SPW}\\ \hline
\end{tabular}

\bigskip

Parallely, there has been a research on universal groups for subclasses of non-Archimedean Polish groups. We refer to the recent article \cite{GaXu} of Gao and Xuan, where they in particular prove that in case of non-Archimedean Polish groups there is no universal object admitting bi-invariant compatible metric. Then they summarize that in case of non-Archimedean Polish groups the last interesting class for which the solution is not known is the class of all locally compact non-Archimedean Polish groups.\\

Still much less explored area is the program of looking for (Polish or just separable) metric groups that are \emph{metrically} universal for certain class of (separable) metric groups. Metrically universal group from some class of separable metric groups contains isometric copies of all groups from that class. The only results from this area known to us (besides the well known results from functional analysis such as that $C(2^\Nat)$ is a metrically universal separable Banach space and metrically universal separable abelian $C^*$-algebra) are the following: the Shkarin's universal abelian Polish group is actually a metric group that is metrically universal for a certain restricted class of abelian groups with invariant metric - this class does not even contain the integers with the standard metric (nor, as a consequence, any separable Banach space), however, it contains all abelian groups with bounded invariant metric; we refer to Niemiec's paper \cite{Nie1} for precise formulations. Niemiec also found a 
metrically universal Boolean Polish group in \cite{Nie2} and in \cite{Nie1} for other exponents as well. Answering Shkarin's question, we in \cite{Do1} produced an abelian separable metric group that is metrically universal for all abelian separable groups with invariant metric.

Let us also mention a negative result of Ozawa (brought to our attention by Pestov) from \cite{Oz} that there is no universal separable $\mathrm{II}_1$ factor, which can be interpreted as a non-existence of a separable metrically universal group with bi-invariant \emph{spherical} metric.\\

Our last main comment concerns structures that are highly homogeneous, or, as it is often said with regard to the Gurarij space, that are of universal disposition. Consider for example the space $C(2^\Nat)$. Even though it is a metrically universal separable metric space and a metrically universal separable Banach (normed) space, it lacks certain homogeneity properties. Already in 1920's, P. Urysohn (\cite{Ur}) constructed a Polish metric space $\Ur$ that not only contains isometric copies of all separable metric spaces, it also has the property that for any finite metric spaces $A,B$ and isometric embeddings $\iota_A:A\hookrightarrow \Ur$ and $\rho:A\hookrightarrow B$ there exists an isometric embedding $\iota_B\supseteq \iota_A:B\hookrightarrow \Ur$ extending $\iota_A$ such that $\iota_A=\iota_B\circ\rho$.

Similarly, Gurarij in 1966 in \cite{Gu} constructed a separable Banach space $\mathbb{G}$ that not only contains linear isometric copies of all separable Banach spaces, it also has the property that for any finite dimensional Banach spaces $A,B$, $\varepsilon>0$ and linear isometric embeddings $\iota_A:A\hookrightarrow \mathbb{G}$ and $\rho:A\hookrightarrow B$ there exists a linear $\varepsilon$-isometric embedding $\iota_B\supseteq \iota_A:B\hookrightarrow \Ur$ extending $\iota_A$ such that $\iota_A=\iota_B\circ\rho$. In contrast with the Urysohn space, the $\varepsilon$ cannot be taken to be $0$.

The universal groups that we construct here also are not only universal, they have extension property of the same flavor. We postpone the precise formulation for later. Let us just mention here that, as in the case of the Gurarij space, this extension property works up to some $\varepsilon>0$.\\

The main results of this paper are the following theorems.
\begin{thm}\label{main}
For any positive real constant $K>0$, there exists a Polish metric group $\mathbb{G}_K$ with bi-invariant metric bounded by $K$ which is of `almost-universal disposition' similarly as the Gurarij space. Moreover, $\mathbb{G}_K$
\begin{itemize}
\item is metrically universal for the class of separable groups equipped with bi-invariant metric bounded by $K$,
\item is isometric to the Urysohn sphere of diameter $K$,
\item is generic among separable groups with bi-invariant metric bounded by $K$.

\end{itemize}

\end{thm}
The precise formulations of `almost-universal disposition' and of being `generic' are somewhat technical, thus we postpone it till Section 1.5. Before that, we shall focus on constructing the group and showing its universality.

Since metric universality is much stronger than topological universality, we thus get the following theorem as a corollary.
\begin{thm}\label{main2}
There exists a universal Polish group admitting compatible bi-invariant metric; resp. Polish group with bi-invariant metric universal for the class of second-countable SIN groups.
\end{thm}
\begin{proof}
Let $\mathbb{G}$ be the metrically universal $\mathbb{G}_1$ from Theorem \ref{main} (any other constant would do as well of course). Then we claim it is the desired universal group. Indeed, let $H$ be any Polish group admitting a compatible bi-invariant metric $d$. We may suppose that $d$ is bounded by $1$. If it were not then we would consider the metric $d'$ defined for any $a,b\in H$ as $$d'(a,b)=\frac{d(a,b)}{1+d(a,b)}$$ which is still compatible and bounded by $1$. Thus $H$ isometrically embeds into $\mathbb{G}$ and an isometric embedding is, of course, in particular a topological embedding.
\end{proof}
One of the auxiliary results that we need in order to prove the main results, which is interesting in its own right and to the best of our knowledge was not known before, is the following theorem.
\begin{thm}
Let $(G,d)$ be a metric group of density $\kappa\geq \aleph_0$ equipped with a bi-invariant metric $d$. Then there exists a supergroup $(G,d)\leq(H,\bar{d})$ with $\bar{d}$ being bi-invariant such that $\bar{d}$ extends $d$ and $H$ has the free group of $\kappa$-many generators as a dense subgroup.
\end{thm}

On the other hand and in contrast with our result on a metrically universal abelian group, we have the following theorem.
\begin{thm}
There is no separable group with bi-invariant metric that contains isometric copies of any separable group with bi-invariant metric.

Moreover, there is no metrically universal locally compact separable group with bi-invariant metric.
\end{thm}
\section{Universal groups}
Before we start, let us roughly describe the ideas behind the construction. We construct a certain ultrahomogeneous metric free group (ultrahomogeneous in the same sense as the Urysohn space, the random graph, the Hall's universal locally finite group, etc., are ultrahomogeneous). This ultrahomogeneous group contains an isometric copy of every finitely generated free group equipped with a bounded rational bi-invariant metric that is also finitely generated. We then define a metric on the set of all bi-invariant metrics on free groups and show that for any bounded bi-invariant metric on some finitely generated free group there is a Cauchy sequence of finitely generated bounded rational metrics that converges to it. This will allow us to prove that the metric completion of the ultrahomogeneous metric group contains an isometric copy of every free group with countably many generators equipped with any bi-invariant bounded metric. The proof will be complete when we show that every group with bi-invariant metric 
can be isometrically embedded into a bigger group with bi-invariant metric which has a free group as a countable dense subgroup.
\subsection{Definitions and preliminaries}
\begin{defin}\label{fingenmetric}
Let $(G,d)$ be a metric group. We say that the metric $d$ is \emph{finitely generated} if there exists a finite set $A_G\subseteq G$ (called the generating set for $d$) such that $1\in A_G$, $A_G=A_G^{-1}$ and for every $a,b\in G$ we have $d(a,b)=\min\{d(a_1,b_1)+\ldots+d(a_n,b_n):n\in \Nat,\forall i\leq n (a_i,b_i\in A_G\wedge a=a_1\cdot\ldots\cdot a_n,b=b_1\cdot\ldots\cdot b_n)\}$. In particular, $G$ is (algebraically) generated by $A_G$.
\end{defin}
Any finitely generated metric is automatically bi-invariant which is not hard to check. Alternatively, the reader may consult \cite{Do1} where this fact was proved. Also, the following fact is clear and stated without a proof.
\begin{fact}\label{trivfact}
Let $(G,d)$ be a metric group with a finitely generated metric. Suppose that for some $a,b\in G$ we have $d(a,b)=d(a_1,b_1)+\ldots+d(a_n,b_n)$ for some $a_1,b_1,\ldots,a_n,b_n\in G$ such that $a=a_1\cdot\ldots\cdot a_n$ and $b=b_1\cdot\ldots\cdot b_n$. Then for any $1\leq i\leq j\leq n$ we have that $$d(a_i\cdot\ldots\cdot a_j,b_i\cdot\ldots\cdot b_j)=d(a_i,b_i)+\ldots+d(a_j,b_j).$$
\end{fact}
Let us remark that the requirement on $d$ being finitely generated in the previous fact is not needed in fact, being bi-invariant is sufficient.

We shall also later use the following fact whose proof we sketched in \cite{Do2}. Before, we recall the notion of a Graev metric on free groups. Let $(A,1,d)$ be a pointed metric space and let $F_A$ be the free group with $A\setminus\{1\}$ as the set of free generators and $1$ as the unit. Denote by $A'$ the set $A\setminus \{1\}$ and set $B=\{1\}\cup A'\prod (A')^{-1}$, where $(A')^{-1}=\{a^{-1}:a\in A'\}$ is the set of formal inverses of $A'$. Extend $d$ to $B$ by setting $d(a^{-1},b^{-1})=d(a,b)$, $d(a^{-1},1)=d(a,1)$, and $d(a^{-1},b)=d(a,1)+d(b,1)$, for any $a,b\in A'$. We have $B\subseteq F_A$ and let $\delta$ be a finitely generated metric on $F_A$ generated by $d$ on $B$. We call $\delta$ the \emph{Graev} metric. It follows that finitely generated metric are generalization of the Graev metric.
\begin{lem}\label{factorlem}
Let $(G,d)$ be a metric group, where $d$ is bi-invariant. Then $d$ is finitely generated iff $(G,d)$ is a factor group with the factor metric of a free group over a pointed finite metric space with the Graev metric quotiented by a finitely generated normal subgroup.
\end{lem}
\begin{proof}
We shall show only the direction that will be used later. For the other direction, we refer to \cite{Do2}.

Suppose that $A\subseteq G$ is the generating set for $d$. View $(A,1_G,d)$ as a pointed metric space and consider the free group $F_A$ with the Graev metric $\delta$ extending $d$ on $A$. Consider the group homomorphism $p:F_A\rightarrow G$ determined by the identity mapping on $A$. Since $A$ algebraically generates $G$, it is a surjection. It follows from the definition of the Graev metric that $p$ is also $1$-Lipschitz.
\end{proof}
We shall construct the universal group using Fra\" iss\' e theoretic methods. It is somewhat similar to the construction of the Urysohn universal metric space - the unique complete separable metric space that contains isometric copies of every finite metric space and every partial finite isometry between its two subsets extends to an autoisometry of the whole space. Even though the original construction of Urysohn (\cite{Ur}) precedes the birth of Fra\" iss\' e theory (\cite{Fr}) it could be described as Fra\" iss\' e-like (we also refer the reader to another construction of the Urysohn space done by Kat\v etov in \cite{Kat}) - at first, one constructs a universal and ultrahomogeneous countable rational metric space, which is the Fra\" iss\' e limit of the class of all finite rational metric spaces, and then takes the completion. Here as well, we at first construct a free group with countably many generators with certain rational bi-invariant metric, which is a Fra\" iss\' e limit of certain class of groups, 
and then we take the completion.

Before beginning the construction, we should therefore make a quick introduction to Fra\" iss\' e theory. For a reader uninitiated to Fra\" iss\' e theory, we recommend either Chapter 7 from \cite{Ho} or lecture notes of David Evans from \cite{Ev}. For a general category-theoretic look on Fra\" iss\' e theory, we refer the reader to \cite{DrGo} and \cite{Kub}. Our use of Fra\" iss\' e theory will be however fairly classical. It will differ from the standard constructions in e.g. Urysohn space, random graph, Hall's universal locally finite group, etc., only in that regard that in our category of structures we restrict to a certain proper subclass of all embeddings between them.

Let $L$ be a countable signature and let $\mathcal{K}$ be some {\it countable} class of finitely generated $L$-structures, resp. isomorphism types of finitely generated $L$-structures. Moreover, let $\mathcal{F}$ be a class of morphisms (embeddings) between structures from $\mathcal{F}$ (which we suppose contains all the isomorphisms and identities). 
\begin{itemize}
\item We say that ($\mathcal{K}$,$\mathcal{F}$) satisfies the {\it joint-embedding property} if for every $A,B\in \mathcal{K}$ there is $C\in \mathcal{K}$ and embeddings $\iota_A:A\hookrightarrow C,\iota_B:\hookrightarrow C$ from $\mathcal{F}$.
\item We say that ($\mathcal{K}$,$\mathcal{F}$) has the {\it amalgamation property} if whenever we have $A,B,C\in \mathcal{K}$ such that $A$ is embeddable into $B$ via some embedding $\phi_B\in \mathcal{F}$ and embeddable into $C$ via some embedding $\phi_C\in \mathcal{F}$, then there exists $D\in \mathcal{K}$ and embeddings of $B$ via $\psi_B\in \mathcal{F}$ into $D$ and of $C$ via $\psi_C\in \mathcal{F}$ into $D$ such that $\psi_C\circ \phi_C=\psi_B\circ \phi_B$.
\item If ($\mathcal{K}$,$\mathcal{F}$) satisfies these properties above then we call it a {\it Fra\" iss\' e class}.

\end{itemize}
 
The following theorem is the most important tool in Fra\" iss\' e theory.
\begin{thm}[Fra\" iss\' e theorem, see \cite{Ho} or \cite{Ev} ]\label{fraissethm}
If ($\mathcal{K}$,$\mathcal{F}$) is a Fra\" iss\' e class then there exists unique up to isomorphism countable structure $\mathbb{K}$, called the Fra\" iss\' e limit of ($\mathcal{K}$,$\mathcal{F}$), with the following properties:
\begin{enumerate}
\item $\mathbb{K}$ is a direct limit of some sequence $K_1\to K_2\to\ldots$, where $K_i$'s are from $\mathcal{K}$ and the `arrows' in the limit are from $\mathcal{F}$.
\item If $\mathbb{A}$ is a direct limit of some system $A_1\to A_2\to\ldots$, where $A_i$'s are from $\mathcal{K}$ and the `arrows' in the limit are from $\mathcal{F}$, then $\mathbb{A}$ embeds into $\mathbb{K}$.
\item If $A_1, A_2\in \mathcal{K}$ are isomorphic and there are embedding $\iota_1:A\rightarrow K_n$, $\iota_2:A_2\rightarrow K_n$, for some $n$ and where $\iota_1,\iota_2\in \mathcal{F}$, then every partial isomorphism between $\iota_1[A_1]$ and $\iota_2[A_2]$ extends to a full autoisomorphism of $\mathbb{K}$.
\item For any $A_1,A_2\in \mathcal{K}$ and embeddings $\iota_1:A_1\rightarrow K_n$, $\rho:A_1\rightarrow A_2$, for some $n$ and where $\iota_1,\rho\in \mathcal{F}$, there exists $m>n$ and an embedding $\iota_2: A_2\rightarrow K_m$ such that $\iota_2\circ \rho=\iota_1$.

\end{enumerate}
\end{thm}
\subsection{Construction of the group}
\begin{defin}[Category $\Gr_K$]\label{GRdefin}
Let $K>0$ be an arbitrary positive rational. Objects of $\Gr_K$ are free groups with finitely many free generators equipped with a finitely generated rational metric bounded by $K$ (i.e. the distance between two elements of the group is the minimum of $K$ and the minimum from Definition \ref{fingenmetric}. More formally, an object is a pair $(F=\{f_1,\ldots,f_n\},d_F)$, where $F$ is a finite set and $d_F$ is a rational finitely generated metric bounded by $K$ on a free group with $F$ as a set of free generators. We shall denote the free group corresponding to this pair by $G_F$.

A morphism between $(F,d_F)$ and $(H,d_H)$ is a group isometric embedding $\iota: G_F\hookrightarrow G_H$ such that for every $f\in F$ we have $\iota(f)\in H$. In particular, $|F|\leq |H|$ and $G_H\cong G_F\ast F_{|H|\setminus |F|}$, where $F_{|H|\setminus |F|}$ is a free group of $|H|\setminus |F|$ generators.
\end{defin}
Our next goal will be to prove that $\Gr$ is a Fra\" iss\' e class which will mainly consist of proving that it is an amalgamation class. Before we do that we introduce some machinery useful for working with metric free groups.

If $F$ is some set, then to each word over the alphabet $F\cup F^{-1}\cup\{1\}$, where $F^{-1}=\{f^{-1}:f\in F\}$ we may associate some element of the free group, denoted again by $G_F$, freely generated by elements of $F$. Formally, we will distinguish between expressions $f_1f_2\ldots f_n$ and $f_1\cdot f_2\cdot\ldots\cdot f_n$, where $f_1,f_2,\ldots,f_n\in F\cup F^{-1}\cup\{1\}$ as the former shall denote some word, while the latter the corresponding group element in $G_F$. Obviously two different words may represent the same corresponding group element and group elements are in one-to-one correspondence with the irreducible words which are words that do not contain subwords of the form $ff^{-1}$, $1f$ or $f1$ for some $f\in F\cup F^{-1}$.

If $F$ is any set, by $W(F)$ we shall denote the set of all irreducible words over the alphabet $F\cup F^{-1}\cup\{1\}$. If $g\in G_F$, i.e. an element belonging to the free group generated by $F$, then by $|g|$ we shall denote the length of the corresponding irreducible word from $W(F)$. If $N$ is a natural number then by $W_N(F)$ we shall denote the elements of $G_F$ such that the corresponding irreducible words from $W(F)$ have length at most $N$.

This is the main theorem of this section.
\begin{thm}
The category $\Gr_K$ is a Fra\" iss\' e class.
\end{thm}
\begin{proof}
Clearly, $\Gr_K$ is countable. We prove the amalgamation property. The joint embedding property is just a special case.\\ Let $(G_0,d_0),(G_1,d_1),(G_2,d_2)\in \Gr$ be free groups with finitely many generators equipped with finitely generated metrics such that there are morphisms $\iota_1: G_0\hookrightarrow G_1$ and $\iota_2 :G_0\hookrightarrow G_2$. Consequently, algebraically we can think of $G_1$ as $G_0\ast H_1$ and of $G_2$ as $G_0\ast H_2$, where $H_1$ and $H_2$ are free groups of finitely many generators. We need to find a group $(G_3,d_3)\in \Gr$ and morphisms $\rho_1 : G_1\hookrightarrow G_3$, $\rho_2: G_2\hookrightarrow G_3$ such that we have $\rho_1\circ \iota_1=\rho_2\circ \iota_2$.

Algebraically, we may define $G_3$ to be $G_0\ast H_1\ast H_2$. $\rho_1:G_0\ast H_1\hookrightarrow G_0\ast H_1\ast H_2$ and $\rho_2: G_0\ast H_2\hookrightarrow G_0\ast H_1\ast H_2$ are then the canonical embeddings. We need to define a finitely generated rational metric on $G_0\ast H_1\ast H_2$ so that $\rho_1$ and $\rho_2$ are also isometries.

Let $A_i\subseteq G_i$, for $i\in \{0,1,2\}$, be the finite set such that the metric $d_i$ is generated by values on this set. To simplify the notation, we shall not distinguish between $A_1$, resp. $A_2$, and $\rho_1[A_1]$, resp. $\rho_2[A_2]$, and $A_0$ and $\rho_2\circ \iota_2[A_0]=\rho_1\circ \iota_1[A_0]$, i.e. for $i\in \{0,1,2\}$, we think about $A_i$ as a subset of $G_0\ast H_1\ast H_2$. Without loss of generality, we assume that $A_0\subseteq A_1,A_2$ (actually, we may assume that $A_0=A_1\cap A_2$). Since $A_1$ and $A_2$ are equipped with metrics ($d_1\upharpoonright A_1$, resp. $d_2\upharpoonright A_2$) that agree on $A_0$, we may consider the metric amalgamation of $A_1$ and $A_2$ over $A_0$ denoted by $A_3$. That shall be the generating set for the metric on $G_0\ast H_1\ast H_2$.

In other words, $A_3=A_1\cup A_2$ and for any $a,b\in A_3$ we define $d'(a,b)$ to be $$d_i\upharpoonright A_i (a,b)$$ if $a,b\in A_i$ for $i\in\{1,2\}$; note that if $a,b\in A_1\cap A_2$ then $d_1\upharpoonright A_1 (a,b)=d_2\upharpoonright A_2 (a,b)$;\\

\noindent otherwise, we define $d'(a,b)$ to be
$$\min\{K,\min\{d_i\upharpoonright A_i(a,x)+d_j\upharpoonright A_j(x,b):$$ $$i\neq j\in \{1,2\},\\a\in A_i,b\in A_j,x\in A_0\}\}.$$\\

Finally, we define a bi-invariant metric $d_3$ on $G_0\ast\ast H_1\ast H_2$ generated by the values of $d'$ on $A_3$. That is, for any $a,b\in G_0\ast H_1\ast H_2$ we define 
\begin{multline}\label{amalmetric}
d_3(a,b)=\min\{K,\min\{d'(a_1,b_1)+\ldots+d'(a_m,b_m):\\m\in \Nat,\forall i\leq m (a_i,b_i\in A_3),a=a_1\cdot\ldots\cdot a_m,b=b_1\cdot\ldots\cdot b_m\}\}.
\end{multline}
Clearly, this is a finitely generated rational metric bounded by $K$ (generated by values on $A_3$). It remains to check that for any $a,b\in G_0\ast H_1$ we have $d_3(a,b)=d_1(a,b)$, i.e. $\rho_1$ is an isometry (we again do not distinguish between $a,b\in G_0\ast H_1$ and $\rho_1(a),\rho_1(b)\in G_0\ast H_1\ast H_2$). Analogously, the same for pairs from $G_0\ast H_2$. We will do just the former as the latter is completely analogous.\\

Let $a,b\in G_0\ast H_1$ be given. Suppose that $d_3(a,b)=d'(a_1,b_1)+\ldots+d'(a_m,b_m)$ for some $a_1,b_1,\ldots,a_m,b_m\in A_3$ such that $a=a_1\cdot\ldots\cdot a_m$ and $1=b_1\cdot\ldots\cdot b_m$. We may suppose that $d'(a_1,b_1)+\ldots+d'(a_m,b_m)\leq K$. We may also suppose that for every $i\leq m$ we have $d'(a_i,b_i)=d_3(a_i,b_i)$. Then we claim that, without loss of generality, we may assume that for every $i\leq m$ we have that either $(a_i,b_i)\in A_1$ and thus $d'(a_i,b_i)=d_1(a_i,b_i)$, or $(a_i,b_i)\in A_2$ and thus $d'(a_i,b_i)=d_2(a_i,b_i)$ (recall again that if $a_i,b_i\in A_1\cap A_2=A_0$ then $d_1(a_i,b_i)=d_2(a_i,b_i)$). Indeed, if for some $i\leq m$ we have that, let us say, $a_i\in A_1$ and $b_i\in A_2$, then by definition $d'(a_i,b_i)=d_1(a_i,x)+d_2(x,b_i)$ for some $x\in A_0$. Then we could replace the pair $(a_i,b_i)$ by three pairs $(a_i,x)$, $(x^{-1},x^{-1})$ and $(x,b_i)$ so that $d'(a_i,b_i)=d'(a_i,x)+d'(x^{-1},x^{-1})+d'(x,b_i)$.

It is then immediate that $d_3(a,b)\leq d_1(a,b)$ since the minimum in the definition of $d_3$ is taken over a greater set. We need to prove the other inequality. To do that, we use a result of Slutsky from \cite{Slu}. To state his result, we introduce some notation. First, analogously as we defined $d'$ to be the amalgam metric on $A_1 \cup A_2$ over $A_0$, we define an an amalgam metric $p:G_1\cup G_2\rightarrow \Rea$ on $G_1\cup G_2$ over $G_0$. Then for any $a,b\in G_3$ we set
\begin{multline}\label{Slutmetric}
\underline{p}(a,b)=\inf\{p(a_1,b_1)+\ldots+p(a_n,b_n):\\a_1\cdot \ldots\cdot a_n=a,b_1\cdot\ldots b_n=b\}.
\end{multline}
Then we have:
\begin{thm}[Slutsky, Theorem 5.10 \cite{Slu}]
$\underline{p}$ is a bi-invariant metric on $G_3$ extending $d_1$ on $G_1$ and $d_2$ on $G_2$.
\end{thm}
It suffices for us to check that $\underline{p}(a,b)\leq d_3(a,b)$. Indeed, then we will have that $d_3(a,b)\leq d_1(a,b)=\underline{p}(a,b)\leq d_3(a,b)$, thus $d_3(a,b)=d_1(a,b)$ and we will be done. However, that immediately follows from the fact that $d_3(a,b)\leq K$ and from the definitions \eqref{amalmetric} and \eqref{Slutmetric}, since in \eqref{Slutmetric} the infimum is taken over a greater set.

This finishes the proof.
\end{proof}
It follows that the class $\Gr_K$ has a Fra\" iss\' e limit denoted here by $G_K$. It is a free group of countably many free generators equipped with a bi-invariant metric $d_K$ bounded by $K$. The following fact characterizes $G_K$. It is a special instance of the Fra\" iss\' e theorem \ref{fraissethm}.
\begin{fact}\label{Gcharac}
$G_K$ is characterized by the following two properties:
\begin{enumerate}
\item It contains as a subgroup every free group of finitely many free generators equipped with a finitely generated rational metric bounded by $K$.
\item If $F\leq G_K$ is a subgroup that belongs to $\Gr_K$ (i.e. free group of finitely many free generators equipped with a finitely generated rational metric) and $H\in \Gr_K$ is another such a group that can be written as $F'\ast H'$, where $F$ and $F'$ are isometrically isomorphic via an isomorphism $\iota:F\rightarrow F'$, then there exists an isometric group monomorphism $\rho:H\hookrightarrow G_K$ such that $\rho\circ \iota=\mathrm{id}_F$.

\end{enumerate}
\end{fact}
The second item is usually called the finite extension property of $G_K$.

Moreover, since $d_K$ is bi-invariant, the group operations are automatically continuous with respect to the topology induced by $d_K$ and extend to the metric completion of $G_K$ denoted here by $\mathbb{G}_K$ (we refer to \cite{Do1} and \cite{Do2} where the same facts were proved or to \cite{Gao} for other information regarding metrics on groups).

The next theorem implies that we may restrict our attention only to those groups that have a free group as a countable dense subgroup.
\begin{thm}\label{free-dense}
Let $(G,d)$ be an arbitrary separable topological group (of countable weight equivalently in this case) equipped with a compatible bi-invariant metric. Then there exists a bi-invariant metric $\bar{d}$ extending $d$, which is defined on $G\ast F_\infty$, where $F_\infty$ is a free group with countably many generators, such that $F_\infty$ is dense in $G\ast F_\infty$.

The same holds true if the metric in consideration is bounded by some $K$.
\end{thm}
In order to prove it, we invoke the following theorem which we proved in \cite{Do2}.
\begin{thm}[see Theorem 0.2 in \cite{Do2}]\label{Graevext}
Let $(G,d_G)$ be a group with bi-invariant metric and let $(X,d_X)$ be a metric space. Suppose that $d'$ is a metric on the disjoint union $G\amalg X$ which extends both $d_G$ and $d_X$, and such that for every $x\in X$ we have $\inf \{d'(g,x):g\in G\}>0$ (equivalently, $G$ is closed in $G\amalg X$). Then $d'$ extends to the bi-invariant metric $\delta$ on $G\ast F(X)$, where $F(X)$ is the free group with $X$ as a set of generators.
\end{thm}
\begin{proof}[Proof of Theorem \ref{free-dense}]
Let $\{h_n:n\in \Nat\}$ be an enumeration of a countable dense subgroup of $G$. By induction, applying Theorem \ref{Graevext} countably many times, we produce $G\ast F_\infty$, where the set of free generators for $F_\infty$ is $\{f_m:m\in \Nat\}$, such that for every $n, k\in\Nat$ there is $m$ such that $d(f_m,h_n)<1/k$. We note that we do not consider $\Nat$ to contain $0$. With Theorem \ref{Graevext}, it is easy to do. Let $F:\Nat\rightarrow \Nat\times \Nat$ be a bijection and let $F_i(n)$, for $n\in \Nat$ and $i\in \{1,2\}$, denote the projection of $F(n)$ on the $i$-th coordinate. At the $n$-th step of the induction, when we have already produced a metric $d_{n-1}$ on $G\ast F_{n-1}$, we define a metric $d'_n\supseteq d_{n-1}$ on $G\ast F_{n-1}\amalg \{f_n\}$, where $f_n$ is a new added point, such that for every $g\in G\ast F_{n-1}$ we have $d'_n(x,g)=1/F_2(n)+d_{n-1}(h_{F_1(n)},g)$. Applying Theorem \ref{Graevext} we get a metric $d_n$ on $G\ast F_{n-1}\ast F_1\cong G\ast F_n$ such that $d_n(f_n,h_{F_
1(n)})=1/F_2(n)$.

It is easy to check that when the induction is finished we get a metric on $G\ast F_\infty$ as desired.

If $d$ on $G$ is bounded by $K$ then we proceed as above. Just after every application of Theorem \ref{Graevext} we decrease the metric so that it is $K$-bounded; i.e. if Theorem \ref{Graevext} gives us a metric $d'_n$ on $G\ast F_n$ we define $K$-bounded $d_n$ on $G\ast F_n$ as follows: for any $a,b\in G\ast F_n$ we put $d_n(a,b)=\min\{K,d'_n(a,b)\}$.
\end{proof}
We now define a certain distance on the set of all metrics on some free group. It will be of great importance both in proving that $\mathbb{G}_K$ is universal and in proving that there is no metrically universal unbounded group. Recall the definition of $|g|$, where $g$ is some element of a free group, that was given in the paragraph following Definition \ref{GRdefin}.
\begin{defin}[Distance between metrics]\label{distmetr}
Let $F_n$ be a free group with $n$ generators coming from the set $\{f_1,\ldots,f_n\}$.

Let $d$ and $p$ be two bi-invariant metrics on $F_n$ and let $\varepsilon>0$ be a positive real. We say that $d$ and $p$ are $\varepsilon$-close, $\dist(d,p)\leq \varepsilon$, if for every $g\in F_n$ we have $\frac{|d(g,1)-p(g,1)|}{|g|}\leq \varepsilon$.

We may then put $\dist(d,p)=\inf\{\varepsilon:\dist(d,p)\leq \varepsilon\}$.
\end{defin}
Though not really important for our purposes, one can readily check that $\dist$ is indeed a metric.
\begin{lem}\label{rationalapprox}
Let $p$ be a finitely generated metric on $F_n$, $n\in \Nat$, and let $\varepsilon>0$. Then there exists a finitely generated rational metric $d$ on $F_n$ such that $\dist(p,d)\leq \varepsilon$.

The same holds if the metrics in consideration are bounded by some rational $K$.
\end{lem}
\begin{proof}
Let $A\subseteq F_n$ be a finite generating set for $p$. Recall that by $\{f_1,\ldots,f_n\}$ we denote the set of free generators of $F_n$, and set $L=\max\{p(f_i,1):i\leq n\}$,  $l=\min\{p(a,b):a\neq b, a,b\in A\}$ and finally $K=\frac{L}{l}$. For every $(a,b)\in A^2$ such that $a\neq b$, choose some rational number $d'(a,b)\in \Rat$ such that $d'(a,b)\in [p(a,b),p(a,b)+\varepsilon/K]$.

Then for every $a,b\in F_m$ we define $$d(a,b)=\min\{d'(a_1,b_1)+\ldots+d'(a_m,b_m):m\in \Nat, \forall i\leq m (a_i,b_i\in A),$$ $$a=a_1\cdot\ldots\cdot a_m,b=b_1\cdot\ldots\cdot b_m\}.$$ This is a finitely generated rational metric. We show that it is $\varepsilon$-close to $p$.

Clearly, $p\leq d$ so we must check that that for every $a\in F_n$ we have $\frac{d(a,1)-p(a,1)}{|a|}\leq \varepsilon$. Suppose that $p(a,1)=p(a_1,b_1)+\ldots+p(a_m,b_m)$, where for every $i\leq m$ we have $a_i,b_i\in A$ and $a=a_1\cdot\ldots\cdot a_m$, $1=b_1\cdot\ldots\cdot b_m$. We claim that $m\leq K\cdot |a|$. Otherwise, write $a$ as $h_1\cdot\ldots\cdot h_{|a|}$, where for $i\leq |a|$ we have $h_i=f_j^\varepsilon$, for some $j\leq n$ and $\varepsilon\in\{1,-1\}$, i.e. $h_i$ is a free generator of $F_n$ or its inverse. Then we would have that $p(h_1,1)+\ldots+p(h_{|a|},1)\leq L\cdot |a|<L\cdot m/K=l\cdot m\leq p(a_1,b_1)+\ldots+p(a_m,b_m)$, a contradiction.

Thus we get $d(a,1)\leq d(a_1,b_1)+\ldots+d(a_m,b_m)\leq p(a,1)+m\cdot \varepsilon/K\leq p(a,1)+ K\cdot |a|\cdot \varepsilon/K=p(a,1)+|a|\cdot \varepsilon$, and we are done.

The assertion about metrics that are bounded by $K$ is then obvious. Just take $\min\{K,d\}$.
\end{proof}
From now on, instead of an arbitrary bound $K$ we shall consider solely $1$-bounded metrics. We do that only for notational reasons to get rid of one variable. The reader can easily modify everything so that it works for any $K>0$.
\begin{defin}
Let $d$ be a bi-invariant $1$-bounded metric on $F_n$ generated by $\{f_1,\ldots,f_n\}$. Let $p$ be a finitely generated $1$-bounded metric on $F_n$ generated by the values of $d$ on $W_N(\{f_1,\ldots,f_n\})$, i.e. for any $a,b\in F_n$ we have $$p(a,b)=\min\{1,\min\{d(a_1,b_1)+\ldots+d(a_m,b_m):m\in \Nat,$$ $$a=a_1\cdot\ldots\cdot a_m,b=b_1\cdot\ldots\cdot b_m,\forall i\leq m (a_i,b_i\in W_N(\{f_1,\ldots,f_n\})\}\}.$$ Then we say that $p$ is an $N$-approximation of $d$. We have that for every $a,b\in W_N(\{f_1,\ldots,f_n\})$ $p(a,b)=d(a,b)$.
\end{defin}
\begin{lem}\label{approxlemma}
Let $d$ be an arbitrary bi-invariant $1$-bounded metric on $F_n$, for some $n$, where we denote the free generators as $\{f_1,\ldots,f_n\}$. Let $N_0\in \Nat$ and $\varepsilon>0$ be arbitrary. Then there exists $N\geq N_0$ such that the $N$-approximation $p$ of $d$ is $\varepsilon$-close to $d$.
\end{lem}
\begin{proof}
Let $N=\max\{N_0,\lceil \frac{1}{\varepsilon} \rceil\}$ and let $p$ be the $N$-approximation of $d$. Then we claim that $p$ is $\varepsilon$-close to $d$. Let $g\in F_n$ be arbitrary. If $|g|\leq N$, then $d(g,1)=p(g,1)$, thus obviously $\frac{|d(g,1)-p(g,1)|}{|g|}\leq \varepsilon$. If $|g|>N$, then $\frac{|d(g,1)-p(g,1)|}{|g|}\leq \frac{1}{|g|}<\frac{1}{N}\leq \varepsilon$, and we are done.
\end{proof}
Fix a now a bi-invariant metric $d$ on $F_\infty$ (with free generators denoted by $(f_n)_n$). We inductively define a sequence $(\varepsilon^d_n)_{n\in\Nat}$ as follows. We set $\varepsilon^d_1=\min\{1/2,2 d(f_1,1)\}$ and for a general $n$, we set $\varepsilon^d_n=\min\{1/2^n,2d(f_n,1),\varepsilon^d_{n-1}\}$.

In the sequel, when the metric $d$ is clear from the context we shall write $(\varepsilon_n)_n$ instead of $(\varepsilon^d_n)_n$.
\begin{prop}\label{extprop}
Take any natural numbers $n\leq m$. Let $p_1$ be a finitely generated metric on $F_n$ and let $p_2$ be a finitely generated metric on $F_m$ such that $p_1\geq p_2\upharpoonright F_n$ and $p_1$ is $\delta$-close to $p_2\upharpoonright F_n$, for some $\delta>0$. Then there exists a finitely generated metric $p$ on $F'_n\ast F_m$, with free generators $\{f'_1,\ldots,f'_n,f_1,\ldots,f_n,\ldots,f_m\}$, such that $p\upharpoonright F'_n\cong p_1\upharpoonright F_n$, $p\upharpoonright F_m\cong p_2\upharpoonright F_m$ and for every $i\leq n$ we have $p(f'_i,f_i)=\delta$.\\
\end{prop}
\begin{proof}
Let $\{f'_1,\ldots,f'_n\}$ denote the free generators of the free group of $n$ generators denoted here by $F'_n$. Let $A'\subseteq F'_n$ be the finite generating set for the metric $p_1$. Similarly, let $\{f_1,\ldots,f_m\}$ denote the free generators of the free group of $m\geq n$ generators denoted by $F_m$. Let $A\subseteq F_m$ be the finite generating set for the metric $p_2$. Consider the free product $F'_n\ast F_m$. We shall define a finitely generated metric on $F'_n\ast F_m$ with the desired properties. Consider the above mentioned sets $A',A$ as subsets of $F'_n\ast F_m$. Note that $A'\cap A=\{1\}$. For $x,y\in A'\cup A$ set 
$$p'(x,y)= \begin{cases} p_1(x,y) & \text{if } x,y\in A',\\
p_2(x,y) & \text{if } x,y\in A,\\
\delta & \text{if }x=(f'_j)^\varepsilon,y=(f_j)^\varepsilon\text{ or }x=(f_j)^\varepsilon,y=(f'_j)^\varepsilon,\\
\text{undefined} & \text{otherwise}.
\end{cases}
$$

Now let $a,b\in F'_n\ast F_m$ be arbitrary. We define $$p(a,b)=\min\{1,\min \{p'(a_1,b_1)+\ldots+p'(a_k,b_k):$$ $$\forall i\leq k (p'(a_i,b_i)\text{ is defined}),a=a_1\cdot \ldots\cdot a_k,b=b_1\cdot\ldots\cdot b_k\}\}.$$

This clearly defines a finitely generated metric (with $A'\cup A$ as a generating set). Note that for $x,y\in A'\cup A$ such that $p'(x,y)$ was undefined we have $$p(x,y)=\min\{p'(x,a)+p'(a,y):(x,a),(a,y)\in\mathrm{dom}(p')\}.$$ We need to check it satisfies the desired properties. It should be clear that for every $i\leq n$ indeed $p(f'_i,f_i)=\delta$. We shall check that for every $a,b\in F_n\subseteq F_m\subseteq F'_n\ast F_m$ we have $p(a,b)=p_2(a,b)$ and similarly that for every $a,b\in F'_n\subseteq F'_n\ast F_m$ we have $p(a,b)=p_1(a,b)$.

We shall use the following notation. For every element $c\in F'_n$, by $\bar{c}$ we shall denote the corresponding element from $F_n\leq F_m$, i.e. the image of $c$ under the map defined by sending each $f'_j$ to $f_j$, for every $j\leq n$. Similarly, for every $c\in F_n$, by $c'$ we shall denote the corresponding element from $F'_n$, i.e. the image of $c$ under the map defined by sending $f_j$ to $f'_j$, for every $j\leq n$.

We are ready to prove the two assertions above.\\

The former is easier. Let $a,b\in F_n$ be arbitrary. Suppose that $p(a,b)<p_2(a,b)$. Then there exists sequences $a_1,\ldots,a_k$ and $b_1,\ldots,b_k$ such that $a=a_1\cdot\ldots\cdot a_k$, $b=b_1\cdot\ldots\cdot b_k$ and for every $i\leq k$ either $(a_i,b_i)\in (A')^2$ or $(a_i,b_i)\in A^2$, or $a_i=(f'_j)^\varepsilon,b_i=(f_j)^\varepsilon$ or $a_i=(f_j)^\varepsilon,b_i=(f'_j)^\varepsilon$, for some $j\leq n$ and $\varepsilon\in \{-1,1\}$; moreover, we have $p(a,b)=p'(a_1,b_1)+\ldots+p'(a_k,b_k)$. However, replacing each $(a_i,b_i)\in (A')^2$ by $(\bar{a}_i,\bar{b}_i)\in A^2$ and every pair $((f_j)^\varepsilon,(f'_j)^\varepsilon)$ or $((f'_j)^\varepsilon,(f_j)^\varepsilon)$, for some $j\leq n$ and $\varepsilon\in \{-1,1\}$, by $(f^\varepsilon_j,f^\varepsilon_j)$ we get new sequences $c_1,\ldots,c_k,d_1,\ldots,d_k\in A$ such that $a=c_1\cdot\ldots\cdot c_k$, $b=d_1\cdot\ldots\cdot d_k$. Then we have that $$p(a,b)=p'(a_1,b_1)+\ldots+p'(a_k,b_k)\geq$$ $$p'(c_1,d_1)+\ldots+p'(c_k,d_k)=p_2(c_1,d_1)+\ldots+p_2(c_
k,d_k)\geq p_2(a,b)$$ and that is a contradiction.\\

We now prove the latter. Since the metrics are bi-invariant it suffices to check that for every $a\in F'_n$ we have $p_1(a,1)=p(a,1)$. Suppose for contradiction that for some $a\in F'_n$ we have $p(a,1)<p_1(a,1)$. So let $a_1,\ldots,a_k$ and $b_1,\ldots,b_k$ be such that $a=a_1\cdot\ldots\cdot a_k$, $1=b_1\cdot\ldots\cdot b_k$ and for every $i\leq k$ either $(a_i,b_i)\in (A')^2$ or $(a_i,b_i)\in A^2$, or $a_i=(f'_j)^\varepsilon,b_i=(f_j)^\varepsilon$ or $a_i=(f_j)^\varepsilon,b_i=(f'_j)^\varepsilon$, for some $j\leq n$ and $\varepsilon\in \{-1,1\}$. We shall show that $p'(a_1,b_1)+\ldots+p'(a_k,b_k)\geq p_1(a,1)$.

Consider a new pair of sequences $c_1,\ldots,c_k$, $d_1,\ldots,d_k$ where for any $i\leq k$
\begin{itemize}
\item $c_i=a_i$ and $d_i=b_i$ if $(a_i,b_i)\in (A')^2\cup A^2$,
\item $c_i=\bar{a}_i$ and $d_i=b_i$ if $a_i=(f'_j)^\varepsilon,b_i=(f_j)^\varepsilon$, for some $j\leq n$ and $\varepsilon\in \{-1,1\}$,
\item $c_i=a_i$ and $d_i=\bar{b}_i$ if $a_i=(f_j)^\varepsilon,b_i=(f'_j)^\varepsilon$, for some $j\leq n$ and $\varepsilon\in \{-1,1\}$.
\end{itemize}
Let $K$ be the cardinality of the set $\{i\leq k: c_i=\bar a_i\text{ or }d_i=\bar b_i\}$.

Let $c=c_1\cdot\ldots\cdot c_k$. Then $c$ can be written as $h_1\cdot h_2\cdot\ldots\cdot h_l$ where for $i$ odd we have $h_i\in F'_n$ and for $i$ even we have $h_i\in F_n$. This decomposition does not need to be unique since $1$ belongs to both $F'_n$ and $F_m$. This is fixed as follows: if, for some $i\leq k$, we have $c_i=1$, then $c_i$ is treated as an element of $F'_n$ if and only if $d_i\in F'_n$.

Correspondingly, we can then write $1=g_1\cdot\ldots\cdot g_l$, where for each $i\leq l$, if $h_i=c_{i_1}\cdot\ldots\cdot c_{i_j}$, then $g_i=d_{i_1}\cdot\ldots\cdot d_{i_j}$ and for $i$ odd we have $g_i\in F'_n$ and for $i$ even we have $g_i\in F_n$.

Since $p(a,1)=p'(a_1,b_1)+\ldots+p'(a_k,b_k)$ and using Fact \ref{trivfact} and the fact that $p((f'_j)^\varepsilon,(f_j)^\varepsilon)=p((f_j)^\varepsilon,(f'_j)^\varepsilon)=\delta$, we have (assuming without loss of generality that $l$ is even) that
\begin{multline}\label{eq1}
p(a,1)\geq p_1(h_1,g_1)+p_2(h_2,g_2)+\ldots+\\p_1(h_{l-1},g_{l-1})+p_2(h_l,g_l)+K\cdot \delta.
\end{multline}
On the other hand, since for every $i$ odd we have $p_2(\bar{h}_i,\bar{g}_i)\geq p_1(h_i,g_i)-|h_i|\cdot \delta$ and $|h_2|+|h_4|+\ldots+|h_l|=|a|-K$, we get
\begin{multline}\label{eq2}
p_1(h_1,g_1)+\ldots+p_2(h_l,g_l)\geq p_2(\bar{h}_1,\bar{g}_1)+\ldots+\\p_2(h_l,g_l)+(|a|-K)\delta\geq p(\bar{a},1)+(|a|-K)\delta.
\end{multline}
Thus it follows from \eqref{eq1} and \eqref{eq2} that $$p(a,1)\geq p(\bar{a},1)+|a|\delta=p_2(\bar{a},1)+|a|\delta\geq p_1(a,1),$$ and that is a contradiction.
\end{proof}

\subsection{The embedding construction}

\noindent Let $d$ be an arbitrary bi-invariant $1$-bounded metric on $F_\infty$. We consider the sequence $(\varepsilon_n)_n$ as in the previous proposition. By $(g_n)_n$ we denote the free generators of $F_\infty$.

Using Lemma \ref{approxlemma} and Lemma \ref{rationalapprox} we can find a rational finitely generated metric $\rho_1$ on $F_1$ ($\cong \Int$) that is $\varepsilon_1$-close to $d\upharpoonright F_1$. By Fact \ref{Gcharac} we can get the isometrically isomorphic copy denoted by $F^1_1$ in $G$ with a generator denoted by $f^1_1\in G$. Then, using Lemmas \ref{approxlemma} and \ref{rationalapprox} again, we can find a rational finitely generated metric $p_2$ on $F_2$ that is $\varepsilon_2$-close to $d\upharpoonright F_2$. In particular, $p_2\upharpoonright F_1$ is $\varepsilon_2$-close to $d\upharpoonright F_1$. Moreover, we may suppose that $\rho_1\geq p_2\upharpoonright F_1$. Using Proposition \ref{extprop} we can obtain a rational finitely generated metric $\rho_2$ on $F^1_1\ast F_2$ such that
\begin{itemize}
\item $\rho_2\upharpoonright F^1_1\cong \rho_1\upharpoonright F_1$,
\item $\rho_2\upharpoonright F_2\cong p_2\upharpoonright F_2$,
\item $\rho_2(f^1_1,g_1)=\varepsilon_1$.

\end{itemize}
Then we can use the extension property from Fact \ref{Gcharac} to extend $F^1_1\leq G$ to $F^1_1\ast F^2_2$ that is isometrically isomorphic to $F^1_1\ast F_2$, where the generators of $F^2_2$ are denoted by $f^2_1,f^2_2$.\\

Suppose we have produced $F^{n-1}_{n-1}\ast F^n_n\leq G$, with generators denoted by $f^{n-1}_1,\ldots,f^{n-1}_{n-1},f^n_1,\ldots,f^n_n$, such that $d_G\upharpoonright F^{n-1}_{n-1}$ is $\varepsilon_{n-1}$-close to $d\upharpoonright F_{n-1}$, $d_G\upharpoonright F^n_n$ is $\varepsilon_n$-close to $d\upharpoonright F_n$ and for every $i\leq n-1$ we have $d_G(f^{n-1}_i,f^n_i)=\varepsilon_{n-1}$.

Then as above, using Lemmas \ref{approxlemma} and \ref{rationalapprox} again, we can find a rational finitely generated metric $p_{n+1}$ on $F_{n+1}$ that is $\varepsilon_{n+1}$-close to $d\upharpoonright F_{n+1}$. In particular, $p_{n+1}\upharpoonright F_n$ is $\varepsilon_{n+1}$-close to $d\upharpoonright F_n$. Moreover, we may suppose that $p_n\geq p_{n+1}\upharpoonright F_n$. Using Proposition \ref{extprop} we can obtain a rational finitely generated metric $\rho_{n+1}$ on $F^n_n\ast F_{n+1}$ such that
\begin{itemize}
\item $\rho_{n+1}\upharpoonright F^n_n\cong p_n\upharpoonright F_n$,
\item $\rho_{n+1}\upharpoonright F_{n+1}\cong p_{n+1}\upharpoonright F_{n+1}$,
\item $\rho_{n+1}(f^n_i,g_i)=\varepsilon_n$ for every $i\leq n$.

\end{itemize}
Then we can use the extension property from Fact \ref{Gcharac} to extend $F^n_n\leq G$ to $F^n_n\ast F^{n+1}_{n+1}$ that is isometrically isomorphic to $F^1_1\ast F_{n+1}$, where the generators of $F^{n+1}_{n+1}$ are denoted by $f^{n+1}_1,f^{n+1}_{n+1}$.\\

When the induction construction is finished, we have obtained countably many Cauchy sequences $(f^n_1)_n, (f^n_2)_n, \ldots$. Indeed, since for every $n$, $\varepsilon_n\leq 1/2^n$, we have $\sum_n \varepsilon_n<\infty$.

Since $\mathbb{G}\supseteq G$, let us denote $f_i\in \mathbb{G}$ the limit of the sequence $(f^n_i)_n$, for every $i\in \Nat$. We claim that the subgroup of $\mathbb{G}$ generated by $\{f_i:i\in \Nat\}$ is isometrically isomorphic to $(F_\infty,d)$. We claim that the isometric isomorphism is the uniquely defined map $\phi$ which sends each $g_i\in F_\infty$ to $f_i\in \mathbb{G}$. Let $a\in F_\infty$ be arbitrary and let $g^{\delta_1} _{a(1)}\ldots g^{\delta_m} _{a(m)}\in W(\{g_i:i\in \Nat\})$ be the unique irreducible word corresponding to $a$, where $m$ is the length of $a$, i.e. for every $i\leq m$, $\delta_i\in \{1,-1\}$ and $a(i)\in \Nat$. Then $\phi(a)=f_{a(1)}^{\delta_1}\cdot\ldots\cdot f_{a(m)}^{\delta_m}$. It follows that we have $$|d_G(\phi(a),1)-d(a,1)|= \lim_{n\to \infty} |d_G((f_{a(1)}^n)^{\delta_1}\cdot\ldots\cdot (f_{a(m)}^n)^{\delta_m},1)-d(a,1)|\leq$$ $$\lim _{n\to \infty} m/2^n=0$$ and we are done.

\subsection{Group structure on the Urysohn sphere}
In \cite{Do2}, answering Vershik's question, we proved that there is a non-abelian group structure on the Urysohn space. It turns out that here we have the following.
\begin{thm}\label{isowithUry}
$\mathbb{G}_1$ is isometric to the Urysohn sphere.
\end{thm}
Recall that the Urysohn sphere $\Ur_1$ is the sphere of diameter $1$ around any point of the Urysohn space, i.e. if $x\in \Ur$ is some point in the Urysohn space then $\Ur_1=\{y:d_\Ur(x,y)=1/2\}$. It is characterized as the unique complete separable ultrahomogeneous metric space of diameter $1$. As with the Urysohn space one can at first construct the rational Urysohn sphere $\Rat\Ur_1$, the unique countable rational ultrahomogeneous metric space of diameter $1$, and then take the completion. We will prove that $G_1$ is isometric to $\Rat\Ur_1$. It will follow that $\mathbb{G}$ is isometric to $\Ur_1$.

We shall use another well-known characterization of $\Rat\Ur_1$. Recall that a Kat\v etov map $f:X\rightarrow \Rea^+$ on some metric space $X$ is a function satisfying for all $x,y\in X$ $|f(x)-f(y)|\leq d_X(x,y)\leq f(x)+f(y)$. The natural interpretation is that $f$ prescribes distances from some new point to points of $X$.
\begin{fact}[follows from Fact \ref{fraissethm} ]\label{charUry}
Let $X$ be a countable rational metric space of diameter $1$. Then $X$ is isometric to $\Rat\Ur_1$ iff for every Kat\v etov map $f:A\rightarrow \Rat^+\cap [0,1]$, where $A$ is a finite subset of $X$, there is a point $x_f\in X$ that realizes $f$, i.e. for every $a\in A$ we have $f(a)=d(x_f,a)$.
\end{fact}
\begin{proof}[Proof of Theorem \ref{isowithUry}.]
We use the characterization from Fact \ref{charUry}. Let $A\subseteq G_1$ be finite and let $f:A\rightarrow \Rat^+\cap [0,1]$ be some Kat\v etov map. Then there exists finitely generated free group with finitely generated rational metric $F$, i.e. $F\in \Gr_1$, such that $A\subseteq F\leq G_1$. We now use a variant of Theorem \ref{Graevext} that we also proved in \cite{Do2}, namely Theorem 2.1 $(2)$ there. It says that we may realize $f$ in $F\ast \Int$ by the generator of the new copy of the integers so that the extended metric on $F\ast \Int$ is still rational and finitely generated, i.e. $F\ast \Int\in \Gr$. However, then by Fact \ref{Gcharac} $(2)$ we may suppose that this extension of $F$ to $F\ast \Int$ actually exists in $G_1$ and thus the generator of this new copy of the integers, denoted by $x_f$, belongs to $G_1$ and we are done by Fact \ref{charUry}.
\end{proof}
\subsection{Almost-universal disposition and genericity}
The goal of this section is to describe the almost-universal disposition property of $\mathbb{G}_1$ (or $\mathbb{G}_K$ for any other constant $K>0$ of course) that characterizes it up to isometric isomorphism. From that characterization, we will be able to prove that from the view of Baire category, $\mathbb{G}_1$ is a generic element.\\

We shall work just with $\mathbb{G}_1$ which we will simply denote by $\mathbb{G}$, and by $G$ we will denote its canonical countable dense subgroup.

Our first goal is to formulate the almost-universal disposition property of $\mathbb{G}$ that also gives a characterization of $\mathbb{G}$ up to isometric isomorphism. We need to introduce one more notion which was already implicitly used in Proposition \ref{extprop}.
\begin{defin}
Let $H$ be a group with bi-invariant metric $d_H$. Let $H_n$ and $H'_n$ both be subgroups of $H$ algebraically isomorphic to the free group of $n$ generators, for some $n\geq 1$, with free generators $h_1,\ldots,h_n$, resp.  $h'_1,\ldots,h'_n$. Then we write $\intdist_H (H_n,H'_n)\leq \varepsilon$ if for every $i\leq n$ we have $d_H(h_i,h'_i)\leq \varepsilon$.
\end{defin}
Note that then for every irreducible word $w$ over the alphabet of $2n$ elements (for the generators and their inverses), if we denote the realization of $w$ in $H_n$ by $v$ and in $H'_n$ by $v'$, then we have $d_H(v,v')\leq \varepsilon\cdot |w|$ by bi-invariance of $d_H$.

Notice also that the definition formally depends on the enumeration of the free generators in both groups. However, in the following theorem where this notion will be used the enumeration of generators will be always clear from the context; alternatively, the expression $\intdist_H (H_n,H'_n)\leq \varepsilon$ may be read that there \emph{exist} free generators $h_1,\ldots,h_n$, resp.  $h'_1,\ldots,h'_n$ such that for every $i\leq n$ we have $d_H(h_i,h'_i)\leq \varepsilon$. Also, we shall usually write $\intdist$ instead of $\intdist_H$ if $H$ is clear from the context.

The next theorem finally states the almost-universal disposition property that characterizes $\mathbb{G}$, and so puts $\mathbb{G}$ in line with objects such as the Urysohn space, the Gurarij space, etc. It will be also crucial later when proving that $\mathbb{G}$ is generic. The proof uses similar methods as the embedding construction of arbitrary free group with bi-invariant metric into $\mathbb{G}$, and is also similar to the proof of uniqueness of the Gurarij space from \cite{KuSo}.
\begin{thm}\label{homog_char}
Let $\mathbb{H}$ be a Polish metric group with bi-invariant metric $d_{\mathbb{H}}$ bounded by $1$. Then $\mathbb{H}$ is isometrically isomorphic to $\mathbb{G}$ if and only if:
\begin{enumerate}
\item free group of countably many generators, denoted by $H$, is dense in $\mathbb{H}$,
\item for every $\varepsilon>\varepsilon '>0$, for every $0\leq m<n$ and $(F_n,\rho)$, where $\rho$ is a bi-invariant metric bounded by $1$, and for every monomorphism $\iota: F_m\rightarrow \mathbb{H}$ such that $\dist ((F_m,\rho\upharpoonright F_m), (\iota [F_m],d_{\mathbb{H}}\upharpoonright \iota[F_m]))<\varepsilon$, there exists a monomorphism $\bar{\iota}: F_n\rightarrow H\subseteq \mathbb{H}$ such that $\intdist(\bar{\iota}(F_m),\iota[F_m])<\varepsilon$ and $\dist((F_n,\rho),(\bar{\iota}[F_n],d_{\mathbb{H}}\upharpoonright \bar{\iota}[F_n]))<\varepsilon '$.

\end{enumerate}
\end{thm}
\begin{proof}
Let us first show that $\mathbb{G}$ satisfies the conditions of Theorem \ref{homog_char}. Obviously, the first condition is satisfied, so let us focus on the latter. Let $\varepsilon>\varepsilon '>0$, $0\leq m<n$, $(F_n,\rho)$, and a monomorphism $\iota: F_m\rightarrow \mathbb{G}$ as in the condition $(2)$ be given. Let $\varepsilon>\delta=\dist ((F_m,\rho\upharpoonright F_m), \iota [F_m])$. Since the free generators of $G$ are dense in $\mathbb{G}$ there exists $G'_m\leq G$ such that $G'_m\in \Gr$ and $\intdist(\iota [F_m], G'_m)<(\varepsilon-\delta)/2$. By bi-invariance, we also get that $\dist(\iota [F_m],G'_m)<(\varepsilon-\delta)/2$ and thus by triangle inequality that $\dist ((F_m,\rho\upharpoonright F_m), G'_m)<(\varepsilon+\delta)/2$. Then we use Lemmas \ref{approxlemma} and \ref{rationalapprox} to get a group $G_n\in \Gr$ algebraically isomorphic to $F_n$ such that $\dist(G_n,(F_n,\rho))<\varepsilon '$. Finally, using Proposition \ref{extprop} and Fact \ref{Gcharac} we can embed $G_n$ into $G$ via 
some isometric monomorphism $\bar{\iota}$ so that $\intdist(\bar{\iota}[G_m],G'_m)<(\varepsilon+\delta)/2$, where $G_m$ is the subgroup of $G_n$ generated by the `first $m$' free generators, and thus by triangle inequality $\intdist(\bar{\iota}[G_m],\iota[F_m])<\varepsilon$. This finishes the proof of one direction.\\

Now we must show conversely that any $\mathbb{H}$ which is a Polish metric group with bi-invariant metric $d_{\mathbb{H}}$ bounded by $1$ satisfying the conditions $(1)$ and $(2)$ is isometrically isomorphic to $\mathbb{G}$. Note that in $\mathbb{G}$ in fact the set of all free generators $(g_n)_n$ of $G$, the countable dense free subgroup, is dense. The same is true for $\mathbb{H}$. Indeed, denote by $(h_n)_n$ the set of free generators of $H$. Since $H$ is dense in $\mathbb{H}$, it suffices to show that $(h_n)_n$ is dense in $H$. Take some $h\in H$ and $\varepsilon>0$. Clearly, $h$ generates an infinite cyclic group $F_h$, i.e. a free group of one free generator $h$. By $(2)$, there is an embedding $\bar\iota :\Int\rightarrow H$ such that $\intdist(\bar\iota[\Int],F_h)<\varepsilon$, i.e. $d_{\mathbb{H}}(h,\bar\iota(1_\Int))<\varepsilon$ and $\bar\iota(1_\Int)$ is a free generator of $H$.

Now by induction we shall construct sequences of finitely generated free subgroups $(G_n)_n$ and $(H_n)_n$ of $\mathbb{G}$, resp. $\mathbb{H}$, of monomorphisms $\phi_n: G_n\rightarrow H_n$, $\psi_n: H_n\rightarrow G_{n+1}$, so that the following conditions are satisfied:\\
\begin{enumerate}[(a)]
\item $G_0=\{1_{\mathbb{G}}\}$, $H_0=\{1_H\}$,
\item for $\phi_n:G_n\rightarrow H_n$ we have that (if $n>0$) $\dist (G_n,\phi_n[G_n])<1/2^n$ and $\intdist (\phi_{n-1}[G_{n-1}],\phi_n[G_{n-1}])<1/2^{n-1}$,
\item for $\psi_n:H_n\rightarrow G_{n+1}$ we have that (if $n>0$) $\dist (H_n,\psi_n[H_n])<1/2^n$ and $\intdist (\psi_{n-1}[H_{n-1}],\psi_n[H_{n-1}])<1/2^{n-1}$,
\item for any $n$ and for any $x\in G_n$ and any $y\in H_n$ we have $d_{\mathbb{G}}(x,\psi_n\circ \phi_n(x))<|x|/2^n$ and $d_{\mathbb{H}}(y,\phi_{n+1}\circ \psi_n(y))<|y|/2^n$,
\item $G_n\subseteq G_{n+1}$ and $H_n\subseteq H_{n+1}$ for all $n$, and $\bigcup_n G_n=G$ and $\bigcup_n H_n=H$.\\

\end{enumerate}

We set $\phi_0$ and $\psi_0$ to be the trivial maps, i.e. sending the unit to the unit, and $G_1$ to be the free group generated by $g_1$.

Suppose now that $\phi_{n-1}, \psi_{n-1},G_n,H_{n-1}$ have been constructed. We use condition $(2)$ for $\mathbb{H}$ with $\varepsilon=1/2^{n-1}$, $\varepsilon '=1/2^n$, $(G_n,d_{\mathbb{G}}\upharpoonright G_n)$ instead of $(F_n,\rho)$, $\psi_{n-1}[H_{n-1}]$ instead of $F_m$ and $\psi^{-1}_{n-1}$ instead of $\iota$. By $\phi_n$ we denote the corresponding monomorphism going from $G_n$ to $H$ (denoted by $\bar{\iota}$ in $(2)$). Finally, we let $H_n$ be the finitely generated free group generated by $\phi_n[G_n]$, $H_{n-1}$ and the free generator $h_n\in H$. Note that $H_n$ is a finitely generated free subgroup of $H$, $H_{n-1}\subseteq H_n$ and $\phi_n: G_n\rightarrow H_n$. Similarly, using condition $(2)$ for for $\mathbb{G}$ with $\varepsilon=1/2^{n-1}$, $\varepsilon '=1/2^n$, $(H_n,d_{\mathbb{H}}\upharpoonright H_n)$ instead of $(F_n,\rho)$, $\phi_n[G_n]$ instead of $F_m$ and $\phi^{-1}_n$ instead of $\iota$, we obtain $\psi_n$ going from $H_n$ to $G$. Then we let $G_{n+1}$ be the finitely generated free 
group generated by $\psi_n[H_n]$, $G_n$ and the free generator $g_n\in G$.

We check that this is as desired. Items $(b),(c)$ follow immediately from the condition $(2)$. We check $(d)$. We shall do it for an arbitrary $n$ and for $y\in H_n$ such that $|y|=1$, i.e. $y$ is a generator. Generalization for $y\in H_n$ such that $|y|>1$ is immediate from bi-invariance from the metric. The case for $x\in G_n$ is analogous. Thus fix $n>0$ and $y\in H_n$ such that $|y|=1$. Looking back how we constructed $\phi_{n+1}$ we see that we have $\intdist (\psi_n^{-1}[\psi_n[H_n]],\phi_{n+1}[\psi_n[H_n]])=\intdist (H_n,\phi_{n+1} \circ \psi_n[H_n])<1/2^n$, which however precisely means that for any generator $h\in H_n$ we have $d_{\mathbb{H}}(h,\phi_{n+1}\circ \psi_n(h))<1/2^n$; in particular for $y$. Finally, $(e)$ follows from the fact that $G_n$, resp. $H_n$, contains $g_n$, resp. $h_n$, and $(g_n)_n$, resp. $(h_n)_n$, are the free generators of $G$, resp. $H$.\\

We are now ready to construct the isometric isomorphism, denoted by $\Phi$, between $\mathbb{G}$ and $\mathbb{H}$. It suffices to define $\Phi$ on the dense subgroup $G$. Take any $g\in G$. It follows from $(b)$, more precisely from the fact that for any $n$, $\intdist (\phi_{n-1}[G_{n-1}],\phi_n[G_{n-1}])<1/2^{n-1}$, that for any $n\geq n_0$, where $n_0$ is such that (the least) $g\in G_{n_0}$, we have $$d(\phi_n(g),\phi_{n+1}(g))<|g|/2^n$$ (again observe that for any generator $w$ we have $d(\phi_n(w),\phi_{n+1}(w))<1/2^n$ by definition of $\intdist$, thus if $g=w_1\cdot\ldots\cdot w_{|g|}$ for some generators or their inverses $w_1,\ldots,w_{|g|}$, then\\ $d(\phi_n(g),\phi_{n+1}(g))\leq d(\phi_n(w_1),\phi_{n+1}(w_1))+\ldots +d(\phi_n(w_{|g|}),\phi_{n+1}(w_{|g|}))<|g|/2^n$). In particular, we have that the sequence $(\phi_n(g))_{n\geq n_0}$ is Cauchy. We define $\Phi(g)$ to be its limit.

First of all, it is clear that $\Phi$ is a homomorphism since it is a limit of monomorphism, thus it preserves the group operations. Next, it is an isometry since by $(b)$ we have $\dist (G_n,\phi_n[G_n])<1/2^n$, thus for any $g\in G$ we have that $d_{\mathbb{G}}(g,1)=\lim_n d_{\mathbb{H}}(\phi_n(g),1)$. Finally, we must check that it is onto, resp. $\Phi[G]$ is dense. To see it, we define $\Psi:\mathbb{H}\rightarrow \mathbb{G}$ analogously as $\Phi$. For any $h\in H$ we define $\Psi(h)$ to be $\lim_n \psi_n(h)$. We again have that $\Psi:\mathbb{H}\rightarrow \mathbb{G}$ is an isometric monomorphism. We need to check that $\Psi\circ \Phi=\mathrm{id}_{\mathbb{G}}$, resp. $\Phi\circ \Psi=\mathrm{id}_{\mathbb{H}}$. Let us do the former, the latter is analogous. Again, it suffices to check it on any $g\in G$. So fix $g\in G$ and we check that $\Psi\circ \Phi(g)=g$. We have that $$\Psi(\Phi(g))=\lim_m \Psi (\phi_m(g))=\lim_m \lim_n \psi_n (\phi_m(g)).$$ For $m=n$ we have by $(d)$ that $d(g,\psi_m(\phi_m(g)))<|g|/
2^m$. By $(c)$, resp. by $\intdist (\psi_{n-1}[H_{n-1}],\psi_n[H_{n-1}])<1/2^{n-1}$, we have that\\ $d(\psi_m(\phi_m(g)),\lim_n \psi_n(\phi_m(g)))<|g|/2^{m-1}$. By triangle inequality, we thus have $$d(g,\lim_n \psi_n(\phi_m(g)))<|g|/2^{m-2}.$$ Taking the limit along $m$ we get that $$d(g,\lim_m \lim_n \psi_n (\phi_m(g)))=d(g,\Psi\circ \Phi (g))=0$$ which is what we were supposed to show. This finishes the proof.
\end{proof}
\begin{thm}\label{genericthm}
$\mathbb{G}$ is generic. More precisely, the set of all bi-invariant metrics (bounded by $1$) whose completion gives $\mathbb{G}$ is co-meager in the set of all bi-invariant metrics (bounded by $1$) on the free group of countably many generators.
\end{thm}
\begin{proof}
Let $D\subseteq [0,1]^{(F_\infty)^2}$ be the space of all bi-invariant metrics bounded by $1$ on $F_\infty$. Easy computation shows that $D$ is a $G_\delta$ subset of $[0,1]^{(F_\infty)^2}$, thus it is a Polish space. We show that $D_{\mathbb{G}}=\{\rho\in D: \overline{(F_\infty,\rho)}=\mathbb{G}\}$ is dense $G_\delta$.

Let us first show that $D_{\mathbb{G}}$ is dense. So let $d\in D$ be an arbitrary bi-invariant metric on $F_\infty$ bounded by $1$,  $A\subseteq (F_\infty)^2$ a finite subset and $\delta>0$ some positive real number. We have to show that there exists $\rho_d\in D_{\mathbb{G}}$ such that for every $(a_1,a_2)\in A$ we have that $|d(a_1,a_2)-\rho_d(a_1,a_2)|<\delta$. By bi-invariance, we may suppose that for every $(a_1,a_2)\in A$ we have that $a_2=1$. However, we just use property $(2)$ of $\mathbb{G}$ from Theorem \ref{homog_char}. Let $m$ be big enough so that $A$ is a subset of $F_m\leq F_\infty$. Let $n=\max\{|a|:(a,1)\in A\}$. Then by $(2)$ we can find a subgroup $G_m\leq G\leq \mathbb{G}$ algebraically isomorphic to $F_m$ such that $\dist((F_m,\rho\upharpoonright F_m),(G_m,d_G))<\delta/n$. Let $\phi:F_m\rightarrow G_m$ denote this algebraic isomorphism.  Then we have that for every $f\in F_m$ we have that $|\rho(f,1)-d(\phi(f),1)|<\delta/n\cdot |f|$. In particular for every $a$ such that $(a,1)\in A$ we 
have $$|\rho(a,1)-d(\phi(a),1)|<\delta/n\cdot |a|\leq \delta.$$ Since $G$ is isomorphic to $F_\infty$, we may thus without loss of generality assume that $\phi$ is the identity and then we are done.\\

It remains to check that $D_{\mathbb{G}}$ is a $G_\delta$ subset of $D$. Again by Theorem \ref{homog_char} $(2)$, it suffices to check that the subset of all those metrics $\rho$ from $D$ satisfying
$$\forall \varepsilon>\varepsilon'>0\,\forall 0\leq m<n\,\forall(F_n,p)\text{ and any monomorphism }\iota:F_m\rightarrow F_\infty$$ $$\text{ such that }\dist((F_m,p),(\iota[F_m],\rho))<\varepsilon\,\exists\text{ monomorphism }\bar{\iota}:F_n\rightarrow F_\infty$$ $$\text{ such that }\intdist(\iota[F_m],\bar{\iota}[F_m])<\varepsilon\text{ and }\dist((F_n,p),(\bar{\iota}[F_n],\rho))<\varepsilon'$$\\ is $G_\delta$. That is not difficult to check. First of all, observe that all the universal quantifiers in the above statement can be assume to quantify just over countable sets. It is clear for those quantifying over reals. Because of Lemmas \ref{rationalapprox} and \ref{approxlemma}, the quantifier over finitely generated free group with bi-invariant metric bounded by $1$ can be taken just over finitely generated free groups with finitely generated rational metric bounded by $1$, and there are of course only countably many of them. Finally, every monomorphism of $F_n$ into $F_\infty$ is determined on the set of 
generators, thus we also have only countably many of them. It follows that it suffices to check that for fixed $\varepsilon>\varepsilon '>0$, $0\leq m<n$, $(F_n,p)$ and monomorphisms $\iota:F_m\rightarrow F_\infty$ and $\bar{\iota}:F_n\rightarrow F_\infty$ we have that $\intdist(\iota[F_m],\bar{\iota}[F_m])<\varepsilon$ and $\dist((F_n,p),(\bar{\iota}[F_n],\rho))<\varepsilon '$ is an open set of metrics $\rho$. It is clear for $\intdist(\iota[F_m],\bar{\iota}[F_m])<\varepsilon$ since that is a condition verified just on generators, thus on a finite sets. However, the second condition $\dist((F_n,p),(\bar{\iota}[F_n],\rho))<\varepsilon '$ can also be checked only on a finite set. That follows from Lemma \ref{approxlemma} which says that for every $(F_n,p)$ there is some $N$ such that if some $\rho$ is close enough to $p$ on the finite set $\{f\in F_m: |f|\leq N\}$, then $\dist((F_n,p),(F_n,\rho))<\varepsilon '$. This finishes the proof.
\end{proof}
\section{Non-universality}
In this section we prove that there is no metrically universal separable group equipped with a bi-invariant metric. This is in contrast with our result in \cite{Do1} for abelian metric groups.

Let us state the theorem.
\begin{thm}\label{nonuniv}
For every separable metric group $(G,d_G)$, where $d_G$ is bi-invariant, there exists a locally compact separable metric group $(H,d_H)$, where $d_H$ is bi-invariant, such that $H$ is not isometrically isomorphic to any subgroup of $G$. Thus there is no metrically universal separable group with bi-invariant metric and no metrically universal locally compact separable group with bi-invariant metric.
\end{thm}

Let us again start with some rough explanation of the ideas. We were able to show that bounded metrics are well-approximated by finitely generated metrics. This is no longer true for unbounded metrics which turns out to be the main reason for non-existence of a universal unbounded bi-invariant metric group. There are only countably many rational finitely generated bi-invariant metrics, however the density of the metric space of bi-invariant metrics on, say, $F_3$ is uncountable.

Recall Definition \ref{distmetr} of distances between metrics. The key tool in proving Theorem \ref{nonuniv} is the following.

\begin{prop}\label{keylemma}
There exists continuum-many discrete bi-invariant metrics on $F_3$, the free group of three generators, such that for some fixed constant $K>0$ each two of them are of distance greater than $K$.
\end{prop}

Suppose the proposition has been proved. We show how to finish the proof of Theorem \ref{nonuniv}. First we need the following claim.

\begin{claim}\label{claim1}
Let $d_0$ and $d_1$ be two bi-invariant metrics on $F_3$, where we denote the generators by $a$, $b$ and $c$, such that $\dist(d_0,d_1)\geq K$, where $K>0$. Suppose we have isometric embeddings $\iota_0: (F_3,d_0)\hookrightarrow (G,d_G)$ and $\iota_1: (F_3,d_1)\hookrightarrow (G,d_G)$, where $G$ is some group with bi-invariant metric $d_G$. Then we have $d_G(\iota_0(a),\iota_1(a))+d_G(\iota_0(b),\iota_1(b))+d_G(\iota_0(c),\iota_1(c))\geq K$.
\end{claim}
\begin{proof}[Proof of the claim.]
Suppose for contradiction that $d_G(\iota_0(a),\iota_1(a))+d_G(\iota_0(b),\iota_1(b))+d_G(\iota_0(c),\iota_1(c))< K$. Since $\dist(d_0,d_1)\geq K$ there exists a word $w\in W(\{a,b,c\})$ such that $\frac{|d_0(w,1)-d_1(w,1)|}{|w|}\geq K$. We have $d_G(\iota_0(w),\iota_1(w))<|w|\cdot K$. That follows from the assumption on distances between $\iota_0(x)$ and $\iota_1(x)$, for $x\in \{a,b,c\}$, and bi-invariance of $d_G$. Then by triangle inequality we get $$|d_G(\iota_0(w),1)-d_G(\iota_1(w),1)|\leq d_G(\iota_0(w),\iota_1(w))<|w|\cdot K,$$ which is a contradiction.
\end{proof}

Now let $(d_\alpha)_{\alpha<\mathfrak{c}}$ be the collection of continuum-many discrete bi-invariant metrics that we got from Proposition \ref{keylemma}. Denote again the generators of $F_3$ by $a$, $b$ and $c$. Suppose that there exists a universal separable group $\mathbb{G}$ with bi-invariant metric $d$. Then for each $\alpha<\mathfrak{c}$ there exists an isometric embedding $\iota_\alpha: (F_3,d_\alpha)\hookrightarrow (\mathbb{G},d)$. Applying Claim \ref{claim1} we get that for any $\alpha\neq \beta <\mathfrak{c}$ we have $d(\iota_\alpha(a),\iota_\beta(a))+d(\iota_\alpha(b),\iota_\beta(b))+d(\iota_\alpha(c),\iota_\beta(c))\geq K$. However, this is not possible since $\mathbb{G}$ is separable. Since the metrics $(d_\alpha)_{\alpha<\mathfrak{c}}$ are discrete, the same argument works for $\mathbb{G}$ locally compact.  This contradiction finishes the proof of Theorem \ref{nonuniv}, so it remains to prove Proposition \ref{keylemma}.\\

For this, we need to generalize the notion of finitely generated metric. Whenever $G$ is a group and $d':G^2\rightarrow \Rea$ is a partial bi-invariant metric on a subset $A\subseteq G^2$ then we may take the greatest (partial) bi-invariant metric that extends $d'$. That is the content of the next definition.
\begin{defin}\label{partialbi}
Let $G$ be a group and $A\subseteq G$ an arbitrary symmetric subset containing $1$ that algebraically generates $G$. We say that a function $d':A^2\rightarrow \Rea$ is a partial bi-invariant metric if
\begin{itemize}
\item $d'(a,a)=0$ for every $a\in A$,
\item $d'(a,b)=d'(b,a)=d'(a^{-1},b^{-1})$ for every $a,b\in A$,
\item for every $a,b\in A$ and every $a_1,b_1,\ldots,a_n,b_n\in A$ such that $a=a_1\cdot\ldots\cdot a_n$, $b=b_1\cdot\ldots\cdot b_n$ we have $d'(a,b)\leq d'(a_1,b_1)+\ldots+d'(a_n,b_n)$.

\end{itemize}

Then we may always extend $d'$ to the bi-invariant metric on the whole $G$ as follows. For every pair $(a,b)\in G^2$ we set $$d(a,b)=\inf\{d'(a_1,b_1)+\ldots+d'(a_n,b_n):$$ $$(a_1,b_1),\ldots,(a_n,b_n)\in A,a=a_1\cdot\ldots\cdot a_n,b=b_1\cdot\ldots\cdot b_n\}.$$
\end{defin}
This is the greatest bi-invariant metric extending $d'$.

\begin{proof}[Proof of Proposition \ref{keylemma}]
Let us denote the three generators of $F_3$ by $a$, $b$ and $c$. Let $A\subseteq F_3$ be a symmetrization of the following subset: $\{1,a,b,c\}\cup\{a^n\cdot b^n\cdot c^n:n\in \{2^k:k\in\Nat\}\}$. For every $x\in 2^\Nat$ we define a partial bi-invariant metric $d'_x:A\rightarrow \Rea$. Then by $d_x$ we shall denote the extension from Definition \ref{partialbi} on the whole $F_3$. We shall guarantee that the distance between $d'_x$ and $d'_y$, for any $x\neq y\in 2^\Nat$, is greater than $1/6$. Then also the distance between the extensions $d_x$ and $d_y$, for any $x\neq y\in 2^\Nat$, will be greater than $1/6$, and we will be done.\\

Let $x\in 2^\Nat$ be arbitrary. We define $d'_x:A\rightarrow \Rea$ as follows:
\begin{itemize}
\item $d'_x(v^\varepsilon,1)=1$ for every $v\in \{a,b,c\}$ and $\varepsilon\in \{1,-1\}$,
\item $d'_x((a^n\cdot b^n\cdot c^n)^\varepsilon,1)=\begin{cases} 2n & \text{if }x(k)=0\\3n-n/2 & \text{if }x(k)=1
\end{cases}$\\ where $k\in \Nat$, $n=2^k$ and $\varepsilon\in \{1,-1\}$,
\item for $u\neq v\in A\setminus\{1\}$ we define $d'_x(u,v)=d'_x(u,1)+d'_x(v,1)$.
\end{itemize}
On the remaining pairs from $A$ the definition of $d'_x$ is clear from the requirements that it is $0$ on the diagonal and symmetric.

Suppose for a moment that we have already checked that $d'_x$ is indeed a partial bi-invariant metric for every $x\in 2^\Nat$. We show that for $x\neq y\in 2^\Nat$, $\dist(d'_x,d'_y)\geq 1/6$. Let $x\neq y\in 2^\Nat$ be given. Let $k\in \Nat$ be such that $x(k)\neq y(k)$, say $x(k)=0$, $y(k)=1$. Then, for $n=2^k$, we have that $d'_x(a^n\cdot b^n\cdot c^n,1)=2n$, while $d'_y(a^n\cdot b^n\cdot c^n,1)=3n-n/2$. Since $|a^n\cdot b^n\cdot c^n|=3n$ we have $\dist(d'_x,d'_y)\geq \frac{|2n-3n+n/2|}{3n}=1/6$. Thus it remains to verify that for every $x\in 2^\Nat$, $d'_x$ is a partial bi-invariant metric and the extension $d_x$ is a discrete metric. However, the latter will follow easily as $\inf\{d_x(v,w):v\neq w\in F_3\}=\min\{d'_x(v,w):v\neq w\in A\}>0$.\\

Let $x\in 2^\Nat$ be given and from now on fixed. It is sufficient to check that for any $n\in \{2^k:k\in\Nat\}$ we have $$\min\{d'_x(u_1,v_1)+\ldots+d'_x(u_m,v_m):u_1,v_1,\ldots,u_m,v_m\in A,$$ $$\forall i\leq m (u_i\neq a^n\cdot b^n\cdot c^n\wedge v_i\neq a^n\cdot b^n\cdot c^n),u_1\cdot\ldots\cdot u_m=a^n\cdot b^n\cdot c^n,$$ $$v_1\cdot \ldots\cdot v_m=1\}\geq 3n-n/2.$$
In other words, there is nothing that forces the distance $d'_x(a^n\cdot b^n\cdot c^n,1)$ to be `small'; so no matter what the values $\{x(k'):k'<\log_2^n\}$ are we can always consistently put $d'_x(a^n\cdot b^n\cdot c^n,1)$ to be $3n-n/2$.

From now on such $n\in \{2^k:k\in\Nat\}$ is fixed. Note that we can suppose that for no $i\leq m$, $u_i$ or $v_i$ from the decomposition above are equal to $(a^{n'}\cdot b^{n'}\cdot c^{n'})^\varepsilon$, where $n'\in \{2^k:k\in\Nat\}$ and $n'>n$ and $\varepsilon\in \{-1,1\}$. This follows from the fact that $d'_x(a^{n'}\cdot b^{n'}\cdot c^{n'},1)\geq 2n'>3n-n/2$. Let $d_{x\upharpoonright n}$ be a finitely generated metric generated by the values of $d'_x$ on the symmetrization of the set $\{1,a,b,c\}\cup \{a^i\cdot b^i\cdot c^i:i\in \{2^k:k\in\Nat\}\wedge i<n\}$ which we shall denote by $A_n$. It follows from the discussion above that it suffices to check that $d_{x\upharpoonright n} (a^n\cdot b^n\cdot c^n,1)\geq 3n-n/2$.

In order to prove it, we use Lemma \ref{factorlem} and a result of Ding and Gao from \cite{DiGa} on computing the Graev metric. We refer the reader to Section 3 from \cite{DiGa} or Section 2.6 from the book \cite{Gao} for more details on techniques used below or for any unexplained (not sufficiently explained) notion.

We briefly recall that for a pointed metric space $(X,1,d_X)$ there is a metric $\delta$ on the free group $F(X)$, called the Graev metric, with generators from the metric space $X$ with the distinguished point $1$ as unit, which is bi-invariant and extends $d_X$. A match (on a set $\{1,\ldots,k\}$) is a function $\theta:\{1,\ldots,k\}\rightarrow \{1,\ldots,k\}$ satisfying $\theta\circ \theta=\mathrm{id}$ and for no $1\leq i<j\leq k$ we have $i<j<\theta(i)<\theta(j)$. The following result of Ding and Gao connects matches with trivial words. We recall that a word $w_1\ldots w_i$ over some alphabet $X\cup X^{-1}\cup\{1\}$, where $X$ is some set, is called \emph{trivial} if the corresponding group element $w'$ in the free group $F(X)$ is $1$.
\begin{lem}[Ding, Gao (Lemma 3.5 in \cite{DiGa})]\label{matchlemma}
For any trivial word $w=w_1\ldots w_i$ over some alphabet there is a match $\theta:\{1,\ldots,i\}\rightarrow \{1,\ldots,i\}$ such that for every $j\leq i$ we have $w_{\theta(j)}=w^{-1}_j$.
\end{lem}
For a word $w=w_1\ldots w_k$ over an alphabet $X\cup X^{-1}\cup\{1\}$ by $w^\theta$ we denote the word $w^\theta_1\ldots w^\theta_k$, where $w_i^\theta=\begin{cases} w_i & \text{if }\theta(i)<i\\
w^{-1}_{\theta(i)} & \text{if }\theta(i)>i\\
1 & \text{if }\theta(i)=i\\
\end{cases}$

We now state the Ding-Gao's result. By $\rho$ in the statement we denote the extension of $d_X$ from $X\cup\{1\}$ to $X\cup X^{-1}\cup\{1\}$ which satisfies
\begin{itemize}
\item for any $x,y\in X$ we have $\rho(x^{-1},y^{-1})=\rho(x,y)=d_X(x,y)$,
\item for any $x,y\in X$ we have $\rho(x,y^{-1})=\rho(y^{-1},x)=d_X(x,1)+d_X(1,y)$,

\end{itemize}
\begin{thm}[Ding, Gao \cite{DiGa}]\label{DingGao}
Let $(X,1,d_X)$ be a pointed metric space and $\delta$ the corresponding Graev metric on the free group $F(X)$. Then for any irreducible word $w=w_1\ldots w_i\in W(X)$ we have $$\delta(w,1)=\min\{\rho(w_1,w_1^\theta)+\ldots+\rho(w_i,w^\theta_i):\theta\text{ is a match}\}.$$
Alternatively, one can say that there exists a trivial word, namely $w^\theta$ for some match $\theta:\{1,\ldots,i\}\rightarrow \{1,\ldots,i\}$, such that $$\delta(w,1)=\rho(w_1,w_1^\theta)+\ldots+\rho(w_i,w_i^\theta).$$
\end{thm}
\begin{remark}
We note (as we were informed by the referee) that the formula for computing the Graev metric was first obtained by Uspenskij in \cite{Us3} and was already implicitly present in \cite{SiUs}.
\end{remark}
Using Lemma \ref{factorlem}, we realize that $(F_3,d_{x\upharpoonright n})$ is the quotient of the free group $F(A'_n)$ over the set $A'_n=\{1\}\cup\{a,b,c\}\cup \{a^{n'} b^{n'}c^{n'}:n'\in\{2^k:k\in\Nat\},n'<n\}$, with the Graev metric $\delta$, by the subgroup $N$ taken with the factor metric. $N$ is the kernel of the map $p:F(A'_n)\rightarrow F_3$, determined by the identity map from $A'_n\subseteq F(A'_n)$ to $A'_n\subseteq F_3$, in which $a\cdot\ldots\cdot b\cdot\ldots\cdot c$ is identified with $a^{n'}b^{n'}c^{n'}$ for all appropriate $n'$. Note that the Graev metric $\delta$ here is the finitely generated metric generated by the values of $d'_{x\upharpoonright n}$ on $A_n$, where $d'_{x\upharpoonright n}$ is the restriction of $d'_x$ to $A_n$. We remark that similar (though less general) quotients in the category of Banach spaces are considered in \cite{BY}, Theorem 2.4, where a Lipschitz-free Banach space over a metric space having a partial linear structure is quotiented by a subspace `generated by 
that linear structure'.

Then combining that with Theorem \ref{DingGao} we get that there exists an \underline{irreducible} word $\bar{w}=\bar{w}_1\ldots \bar{w}_j$ over the alphabet $A_n=\{a,b,c\}\cup \{a^{n'} b^{n'}c^{n'}:n'\in\{2^k:k\in\Nat\},n'<n\}\cup \{a,b,c\}^{-1}\cup \{a^{n'} b^{n'}c^{n'}:n'\in\{2^k:k\in\Nat\},n'<n\}^{-1}$ and a trivial word $\bar{v}=\bar{v}_1\ldots \bar{v}_j$ over the same alphabet such that $\bar{w}'=a^n\cdot b^n\cdot c^n$, where by $\bar{w}'$ we denote the group element corresponding to the word $\bar{w}$ in the quotient $F(A'_n)/N=F_3=F(\{a,b,c\})$, and we have
\begin{equation}\label{DGsum}
d_{x\upharpoonright n}(a^n\cdot b^n\cdot c^n,1)=d'_{x\upharpoonright n}(\bar{w}_1,\bar{v}_1)+\ldots+d'_{x\upharpoonright n}(\bar{w}_j,\bar{v}_j).
\end{equation}

We shall work with certain modifications of words $\bar{w}$ and $\bar{v}$, denoted by $w=w_1\ldots w_i$ and $v=v_1\ldots v_i$, such that for each $j\leq i$ we have that the pair $(w_j,v_j)$ is one of the following:
\begin{enumerate}
\item $(x^\varepsilon,x^\varepsilon)$, where $x\in \{a,b,c\}$ and $\varepsilon\in \{1,-1\}$,
\item $(x^\varepsilon,1)$ or $(1,x^\varepsilon)$, where $x\in \{a,b,c\}$ and $\varepsilon\in \{1,-1\}$,
\item $((a^{n'} b^{n'} c^{n'})^\varepsilon,1)$, where $n'=2^k$ for some $k$, $n'<n$ and $\varepsilon\in \{1,-1\}$.

\end{enumerate}
and moreover, the sum $d'_{x\upharpoonright n}(w_1,v_1)+\ldots+d'_{x\upharpoonright n}(w_i,v_i)$ is equal to the sum $d'_{x\upharpoonright n}(\bar{w}_1,\bar{v}_1)+\ldots+d'_{x\upharpoonright n}(\bar{w}_j,\bar{v}_j)$.

The words $w$ and $v$ are obtained from $\bar{w}$, resp. $\bar{v}$ as follows: Each pair $(x,y)$, where $x\neq y\in A_n$, is replaced by two pairs $(x,1)$ and $(1,y)$ without changing the sum \eqref{DGsum}. The pair $((a^{n'} b^{n'} c^{n'})^\varepsilon,(a^{n'} b^{n'} c^{n'})^\varepsilon)$, where $n'=2^k$ for some $k$, $n'<n$ and $\varepsilon\in (1,-1)$, is replaced by $3n'$ pairs $(a,a),\ldots,(b,b),\ldots,(c,c)$, resp. $(c^{-1},c^{-1}),\ldots,(b^{-1},b^{-1}),\ldots,(a^{-1},a^{-1})$ again without changing the sum \eqref{DGsum}. The pair $(1,(a^{n'} b^{n'} c^{n'})^\varepsilon)$, where $n'=2^k$ for some $k$, $n'<n$ and $\varepsilon\in \{1,-1\}$, is replaced by $3n'+1$ pairs $((a^{n'} b^{n'} c^{n'})^{-\varepsilon},1)$ and then $(a,a),\ldots,(b,b),\ldots,(c,c)$ or\\ $(c^{-1},c^{-1}),\ldots,(b^{-1},b^{-1}),\ldots,(a^{-1},a^{-1})$ again without changing the sum \eqref{DGsum}.

Next, associate to the word $w$ the canonical corresponding word $u=u_1\ldots u_l$ over the alphabet $\{1,a,b,c\}\cup \{a,b,c\}^{-1}$; i.e. each letter $w_j$, for $j\leq i$, which is of the form $(a^{n'} b^{n'} c^{n'})^\varepsilon$ is replaced by the corresponding word from the alphabet $\{1,a,b,c\}\cup \{a,b,c\}^{-1}$. In particular, we have that $u'$, the corresponding group element in $F_3$, is equal to $a^n\cdot b^n\cdot c^n$. Thus there exists a subsequence $l_0=1\leq l_1<l_2<\ldots l_{3n}\leq l=l_{3n+1}$, where
\begin{itemize}
\item $u_{l_j}=a$ for $1\leq j\leq n$, $u_{l_j}=b$ for $n+1\leq j\leq 2n$, $u_{l_j}=c$ for $2n+1\leq j\leq 3n$,
\item for every $0\leq j\leq 3n$, $u_{l_j+1}\ldots u_{l_{j+1}-1}$ is a trivial word, i.e. the corresponding group element is $1$.

\end{itemize}
It follows (using Lemma \ref{matchlemma}) that for each $j\leq 3n$ there exists a partial match $\eta_j:\{l_j+1,\ldots,l_{j+1}-1\}\rightarrow \{l_j+1,\ldots,l_{j+1}-1\}$ that is a match for the trivial subword $u_{l_j+1}\ldots u_{l_{j+1}-1}$ in $u$. By $\eta:\{1,\ldots,l\}\rightarrow \{1,\ldots,l\}$ we denote the union $\bigcup_j \eta_j\cup \{(l_j,l_j):j\leq 3n\}$.

Similarly, by $t=t_1\ldots t_l$ we denote the trivial word of the same length as $u$ that corresponds to $v$. Note that already $v$ was a word over the alphabet $\{1,a,b,c\}\cup \{a,b,c\}^{-1}$, so $t$ is only a lengthening; i.e. for some $j\leq i$, if $v_j\in \{a,b,c\}\cup\{a,b,c\}^{-1}$, then the corresponding letter also appears in $t$, and if $v_j=1$ then $t$ contains $|w_j|$-many $1$'s at that place, where $|w_j|$ is the length of $w_j$ as a word over the alphabet $\{1,a,b,c\}\cup \{a,b,c\}^{-1}$.

So since $t$ is trivial, let us also denote by $\theta:\{1,\ldots,l\}\rightarrow \{1,\ldots,l\}$ a match for $t$. Also, for each $j\leq l$ let $\phi(j)$ be the index such that $u_j$ `comes from' the letter $w_{\phi(j)}$ (each $w_j$ can be considered as a subword of $u$). Note that for every $j\leq 3n$ either $w_{\phi(l_j)}$ is the same letter as $u_{l_j}$, or $w_{\phi(l_j)}$ is the letter $a^{n'}b^{n'}c^{n'}$, for some $n'$, from $A_n$.

We shall now define
\begin{enumerate}
\item a one-to-one map $\tau:\{l_1,\ldots,l_{3n}\}\rightarrow \tau[\{l_1,\ldots,l_{3n}\}]$,
\item a `cost function' $\pi:\{l_1,\ldots,l_{3n}\}\rightarrow \Rea$ which will help us estimate the value $d_{x\upharpoonright n}(a^n\cdot b^n\cdot c^n,1)$ in the sense that we will have $$\sum_{j=1}^{3n} \pi(l_j)\leq d'_{x\upharpoonright}(w_1,v_1)+\ldots+d'_{x\upharpoonright n}(w_i,v_i),$$
\item a function $\alpha:\{l_1,\ldots,l_{3n}\}\rightarrow \{0,1,\ldots,i\}$ whose meaning will be explained during the definition.

\end{enumerate}

For every $j\leq 3n$, if $v_{l_j}=1$ then 
\begin{itemize}
\item we set $\tau(l_j)=l_j$,
\item if $w_{\phi(l_j)}$ is the same letter as $u_{l_j}$, i.e. $w_{\phi(l_j)}\in \{a,b,c\}$, then we set $\alpha(l_j)=0$; in that case, the sum \eqref{DGsum} contains as a summand $d'_{x\upharpoonright n}(w_{\phi(l_j)},1)=d'_{x\upharpoonright n}(u_{l_j},1)=1$, so we set `the cost' $\pi(l_j)=1$ to reflect this;
\item if on the other hand $w_{\phi(l_j)}$ is of the form $a^{n'}b^{n'}c^{n'}$, for some $n'$, then we set $\alpha(l_j)=\phi(l_j)$ to denote that the letter $u_{l_j}$ is part of $a^{n'}b^{n'}c^{n'}=w_{\phi(l_j)}$; the cost is estimated as follows: let $C=|\{j'\leq 3n: \alpha(l_{j'})=\alpha(l_j)\}|$ and set $$\pi(l_j)=\frac{d'_{x\upharpoonright n}(w_{\phi(l_j)},1)} {C}=\frac{d'_{x\upharpoonright n}(a^{n'}b^{n'}c^{n'},1)}{C}.$$

\end{itemize}

If on the other hand we have that $v_{l_j}\neq 1$, then we necessarily have that $v_{l_j}=u_{l_j}$. Then using the match $\theta$ we have that $v_{\theta(l_j)}=v^{-1}_{l_j}$ and so also $u_{\theta(l_j)}=u^{-1}_{l_j}$. Then we use the match $\eta$ to get $u_{\eta\circ\theta(l_j)}$.
\begin{claim}
The letter $u_{\eta\circ\theta(l_j)}$ is a part of the letter from $w$ that is of the form $a^{n'}b^{n'}c^{n'}$, for some $n'$, i.e. $w_{\phi\circ\eta\circ\theta(l_j)}$ is of this form.
\end{claim}
To prove it, assume that $\theta(l_j)<\eta\circ\theta(l_j)$. The subword\\ $u_{\theta(l_j)+1}\ldots u_{\eta\circ\theta(l_{j+1})-1}$ is trivial. Since the word $\bar{w}$ is an irreducible word over the alphabet $A_n$, the subword $u_{\theta(l_j)+1}\ldots u_{\eta\circ\theta(l_{j+1})-1}$ contains one of those:
\begin{itemize}
\item the subword $a\ldots a b\ldots b c\ldots c$ followed by the subword\\ $c^{-1}\ldots c^{-1} b^{-1}\ldots b^{-1} a^{-1}\ldots a^{-1}$, where the former is a part of the letter from $w$ that is of the form $a^{m}b^{m}c^{m}$, for some $m$; or
\item the subword $c^{-1}\ldots c^{-1} b^{-1}\ldots b^{-1} a^{-1}\ldots a^{-1}$ followed by the subword $a\ldots a b\ldots b c\ldots c$, where the former is a part of the letter from $w$ that is of the form $c^{-m}b^{-m}a^{-m}$, for some $m$.
\end{itemize}
Assume the former (the latter is treated analogously). Then if $w_j=a^{m}b^{m}c^{m}$ is the letter from $w$ corresponding to that subword, then $v_j=1$. Moreover, for each $u_j$ that is a part of the subword\\ $c^{-1}\ldots c^{-1} b^{-1}\ldots b^{-1} a^{-1}\ldots a^{-1}$, i.e. $u_j\in\{a^{-1},b^{-1},c^{-1}\}$, we must have that $v_j=u_j$. If not, then on the corresponding subwords of $w$, resp. $v$, $w_j\ldots w_{j'}$ and $v_j\ldots v_{j'}$ the metric $d_{x\upharpoonright n}$ would not be minimizing; i.e. we could find pairs $(x_1,y_1),\ldots,(x_k,y_k)$ such that $(x_1\ldots x_k)'=(u_{l_j+1}\ldots u_{l_{j+1}-1})'=1$, $(y_1\ldots y_k)'=(v_{l_j+1}\ldots v_{l_{j+1}-1})'$, however\\ $d_{x\upharpoonright n}(x_1,y_1)+\ldots+d_{x\upharpoonright n}(x_k,y_k)$ would be smaller than the sum over the corresponding letters from $w$ and $v$ respectively.

However, then we could assume that the pairs\\ $(a^{m}b^{m}c^{m},1)$,$(c^{-1},c^{-1}),\ldots,(b^{-1},b^{-1}),\ldots,(a^{-1},a^{-1})$ were obtained from the pair $(1,c^{-m}b^{-m}a^{-m})$ in the words $\bar{w},\bar{v}$. That would be a contradiction with the fact that $\bar{w}$ were irreducible as $\bar{w}$ would contain $1$.\\
  
In this case we set 
\begin{itemize}
\item $\tau(l_j)=\eta\circ\theta(l_j)$,
\item $\alpha(l_j)=\phi\circ\eta\circ\theta(l_j)$ to denote that $u_{l_j}$, via `matching', belongs to $w_{\phi\circ\eta\circ\theta(l_j)}$,
\item the cost is then, similarly as above, estimated as follows:  let $C=|\{j'\leq 3n: \alpha(l_{j'})=\alpha(l_j)\}|$ and set $$\pi(l_j)=\frac{d'_{x\upharpoonright n}(w_{\phi\circ\eta\circ \theta(l_j)},1)} {C}=\frac{d'_{x\upharpoonright n}(a^{n'}b^{n'}c^{n'},1)}{C}.$$

\end{itemize}
It is straightforward to check that the cost function $\pi$ gives the estimate 
\begin{equation}\label{last_est}
\sum_{j=1}^{3n} \pi(l_j)\leq d'_{x\upharpoonright}(w_1,v_1)+\ldots+d'_{x\upharpoonright n}(w_i,v_i).
\end{equation}

\begin{lem}\label{lastlemma}
There can be at most one $i'\neq 0$, $i'\in \{1,\ldots,i\}$ such that for some $1\leq j\leq n$, $n+1\leq j'\leq 2n$, $2n+1\leq j''\leq 3n$, i.e. $u_{l_j}=a$, $u_{l_{j'}}=b$ and $u_{l_{j''}}=c$, we have that $\alpha(l_j)=\alpha(l_{j'})=\alpha(l_{j''})=i'$.
\end{lem}
Suppose otherwise that there are $i'<i''\neq 0$, $i',i''\in \{1,\ldots,i\}$ such that for some $1\leq j\neq k\leq n$, $n+1\leq j'\neq k'\leq 2n$, $2n+1\leq j''\neq k''\leq 3n$ we have that $\alpha(l_j)=\alpha(l_{j'})=\alpha(l_{j''})=i'$ and $\alpha(l_k)=\alpha(l_{k'})=\alpha(l_{k''})=i''.$ Note that the letter $w_{i'}$ corresponds to some interval of letters from the word $u$ (of the type $a^{n'}b^{n'}c^{n'}$ for some $n'$), the same for the letter $w_{i''}$; i.e. there are disjoint intervals $I'<I''\subseteq \{1,\ldots,l\}$ such that $(u_j)_{j\in I'}$ correspond to $w_{i'}$ and $(u_j)_{j\in I''}$ correspond to $w_{i''}$. Moreover, by definition of the function $\tau$, we have that $\tau(l_j),\tau(l_{j'}),\tau(l_{j''})\in I'$ and  $\tau(l_k),\tau(l_{k'}),\tau(l_{k''})\in I''$, thus $\tau(l_j)<\tau(l_{j'})<\tau(l_{j''})<\tau(l_k)<\tau(l_{k'})<\tau(l_{k''})$. This, together with the fact that $\{l_j,l_k\}<\{l_{j'},l_{k'}\}<\{l_{j''},l_{k''}\}$ (where inequality between sets means that each element from one set 
is smaller than every element from the other set), will lead us to a contradiction. We shall need two claims.
\begin{claim}
For no pair $x\in \{l_j,l_{j'},l_{j''}\}$, $y\in \{l_k,l_{k'},l_{k''}\}$ we can have any of those inequalities
\begin{equation}\label{almostmatch}
\begin{split}
x<y<\tau(x)<\tau(y),\\
\tau(x)<\tau(y)<x<y,\\
x<\tau(y)<\tau(x)<y,\\
\tau(y)<x<y<\tau(x).
\end{split}
\end{equation}
i.e. $\tau$ behaves like a match for such a pair.
\end{claim}
The verification is tedious, however it is rather uneventful. It helps to draw a picture of the situation. We show that we cannot have $$l_j<l_k<\tau(l_j)<\tau(l_k).$$ The other cases are completely analogous. Note that we have $\tau(l_j)=\eta\circ\theta(l_j)$ and $\tau(l_k)=\eta\circ\theta(l_k)$, where $\theta$ is a match and $\eta$ is a match of the form $\bigcup_j \eta_j\cup \{(l_j,l_j):j\leq 3n\}$. First of all, since $\eta$ and $\theta$ are matches, in order to have $l_j<l_k<\eta\circ\theta(l_j)<\eta\circ\theta(l_k)$ we must have $l_j<l_k<\theta(l_k)<\theta(l_j)$. Recall that for each $r\leq 3n$, $\eta_r$ is a match on the trivial word $u_{l_r+1}\ldots u_{l_{r+1}-1}$ which we shall denote by $U_r$ here. Thus if we are to have $\eta\circ\theta(l_j)<\eta\circ\theta(l_k)$ there must be $r\leq 3n$ such that both $\theta(l_j)$ and $\theta(l_k)$ belong to the trivial word $U_r$, and so $\eta\circ\theta(l_j)=\eta_r\circ\theta(l_j)$ and $\eta\circ\theta(l_k)=\eta_r\circ\theta(l_k)$. We can write $U_r$ as $$s_1 
u_{\theta(l_k)}=a^{-1} s_2 u_{\theta(l_j)}=a^{-1} s_3 w_{i'}=a^{n'}b^{n'}c^{n'} s_4 a=u_{\tau(l_k)} s_5$$ where $s_1 s_5$ is trivial, $s_3$ is trivial and $s_2 s_4$ is trivial. Since $u_{\tau(l_{j'})}$ and $u_{\tau(l_{j''})}$ are part of $w_{i'}$ we have that $u_{\theta(l_{j'})}$ is either part of $s_2$ or part of $s_4$; the same is true for $u_{\theta(l_{j''})}$. However, none of them can be part of $s_2$. Indeed, suppose that, let us say, $u_{\theta(l_{j'})}$ is part of $s_2$. Then we have that $\theta(l_k)<\theta(l_{j'})<\theta (l_j)$ and since $\theta$ is a match, we have that $l_j<l_{j'}<l_k$ which is a contradiction. The verification for $u_{\theta(l_{j''})}$ is the same. Thus $u_{\theta(l_{j'})}$ and $u_{\theta(l_{j''})}$ are part of $s_4$, and since $\tau(l_{j'})<\tau(l_{j''})<\min\{\theta(l_{j'}),\theta(l_{j''})\}$, we have that $\theta(l_{j''})<\theta(l_{j'})$. Now since $\theta(l_{j''})<\theta(l_{j'})<\tau(l_k)$ and necessarily also $\tau(l_k)<\tau(l_{k'})<\tau(l_{k''})<\min\{\theta(l_{k'}),\
theta(l_{k''})\}$, we have $$\theta(l_{j''})<\theta(l_{j'})<\theta(l_{k''})<\theta(l_{k'}).$$ Since $\theta\circ\theta(l_{j''})=l_{j''}>\theta(l_{k'})$ and $\theta$ is a match we get that $$l_{k'}<l_{k''}<l_{j'}<l_{j''}$$ which is a contradiction.\\

\begin{claim}
We have \begin{enumerate}
\item either $l_k<l_j\leq \tau(l_j)<\tau(l_k)$,
\item or $l_{j''}>l_{k''}\geq\tau(l_{k''})>\tau(l_{j''})$.

\end{enumerate}
\end{claim}

Suppose that the first case does not happen, however $l_j\leq\tau(l_j)$. Then since $\tau(l_k)>\tau(l_j)$ we necessarily have that $l_k>\tau(l_j)$ because of \eqref{almostmatch}. Since $\tau(l_j)<\tau(l_{j''})<\tau(l_{k''})$ and $\tau(l_j)<\min\{l_{j''},l_{k''}\}$ (since $l_k<\min\{l_{j''},l_{k''}\}$) we necessarily have that $\tau(l_{j''})<\tau(l_{k''})< l_{k''}<l_{j''}$ (again using the `almost-match' property of $\tau$ from \eqref{almostmatch}), and we are in the second case.

Thus suppose that neither the first nor the second case happen and $\tau(l_j)<l_j$. Then necessarily also $\tau(l_{j''})<l_{j''}$. Since the second case does not happen we must have that $l_{j''}<l_{k''}$. Because of \eqref{almostmatch}, we must have that $l_{j''}<\tau(l_{k''})$. However, since $l_k<l_{j''}$ this implies that $\tau(l_{j''})<l_{k'}<l_{j''}<\tau(l_{k'})$ which contradicts \eqref{almostmatch}.

So necessarily either the first or second case happens. Suppose the first case happens. We shall reach a contradiction. The seconds case is symmetric to the first one, so the contradiction can be reached analogously. We necessarily have that $l_{k'}>\tau(l_j)$ since otherwise we would have $l_j<l_{k'}<\tau(l_j)<\tau(l_{k'}$ which would contradict \eqref{almostmatch}. Also, we necessarily have that $l_{j''}<\tau(l_k)$ since otherwise we would have that $l_k<\tau(l_{j''})<\tau(l_k)<l_{j''}$ which would again contradict \eqref{almostmatch}. However then we have that $\tau(l_{j''})<l_{k'}<l_{j''}<\tau(l_{k'})$ which again contradicts \eqref{almostmatch}.\\

We are ready to finish the proof. By Lemma \ref{lastlemma} there is at most one $i'$ such that for some $1\leq j\leq n$, $n+1\leq j'\leq 2n$, $2n+1\leq j''\leq 3n$, i.e. $u_{l_j}=a$, $u_{l_{j'}}=b$ and $u_{l_{j''}}=c$, we have that $\alpha(l_j)=\alpha(l_{j'})=\alpha(l_{j''})=i'$. The worst we can expect is that
\begin{itemize}
\item $w_{i'}=a^{n'}b^{n'}c^{n'}$, where $n'=n/2$ and $d'_{x\upharpoonright n}(a^{n'}\cdot b^{n'}\cdot c^{n'},1)=2n'$,
\item for $n'$-many $a$'s, $n'$-many $b$'s and $n'$-many $c$'s we have that they are associated to $w_{i'}$; i.e. more precisely, there is a set $J\subseteq[1,\ldots,n]$ such that $|J|=n'$ and for every $j\in J$ we have $\alpha(l_j)=i'$; similarly there are sets of size $n'$, $J'\subseteq[n+1,\ldots,2n]$ and $J''\subseteq[2n+1,\ldots,3n]$ such that for every $j\in J'\cup J''$ we have $\alpha(l_j)=i'$.

\end{itemize}
We then have that $$\sum_{j\in J\cup J'\cup J''} \pi(l_j)=3n'\cdot \frac{2n'}{3n'}=2n'.$$
It follows that for any other $i''\neq 0$ in the range of $\alpha$, i.e. corresponding to some word $w_{i''}$ of the form $a^{n''}b^{n''}c^{n''}$ we can have at most $2n''$ associated $j$'s from $[1,\ldots,3n]$ such that $\alpha(l_j)=i''$. For any such $j$ the cost is thus $$\pi(l_j)=\frac{d'_{x\upharpoonright n}(a^{n''}\cdot b^{n''}\cdot c^{n''},1)}{|\{j':\alpha(j')=i''\}|}\geq \frac{2n''}{2n''}=1.$$ Since also for any $j\in [1,\ldots,3n]$ such that $\alpha(l_j)=0$ we have $\pi(l_j)=1$,
it follows that we have that $$\sum_{j\in [1,\ldots,3n]\setminus (J\cup J'\cup J'')} \pi(l_j)\geq |[1,\ldots,3n]\setminus (J\cup J'\cup J'')|=3n'.$$
Combining that with \eqref{last_est} we get $$d'_{x\upharpoonright}(w_1,v_1)+\ldots+d'_{x\upharpoonright n}(w_i,v_i)\geq \sum_{j=1}^{3n} \pi(l_j)\geq 2n'+3n'=3n-n/2.$$ This finishes the proof.
\end{proof}
\section{Concluding remarks and open problems}
Let us conclude with a remark on how to construct analogous universal groups of higher densities (for certain cardinals) and how the situation there differs from the countable case. Then we proceed to several problems and questions that arose during our research.
\subsection{Universal groups of uncountable weight}
We briefly describe how to generalize the construction for certain uncountable density, resp. weight. Let $\kappa$ be an uncountable cardinal such that $\kappa^{<\kappa}=\kappa$. Consistently, there is no such cardinal. On the other hand, consistently, assuming the generalized continuum hypothesis, every successor cardinal has this property. It is known that under this assumption, the generalized Fra\" iss\' e theorem holds, i.e. for any class of structures generated by less than $\kappa$-many elements which is of size at most $\kappa$ and has the joint-embedding and amalgamation properties, there is a corresponding Fra\" iss\' e limit. Moreover, the following observation will be useful.
\begin{observation}\label{observ_limit}
Assume the generalized continuum hypothesis. Let $L$ be some, say countable, signature. Let $\lambda$ be an uncountable limit cardinal and let $\Age_\lambda$ be some class of $L$-structures generated by at most $\lambda$-many elements which is closed under taking direct limits of directed systems from $\Age_\lambda$ of size less than $\lambda$. For every infinite successor cardinal $\kappa^+<\lambda$, let $\Age_\kappa$ be the subclass of $\Age_\lambda$ consisting of $L$-structures generated by at most $\kappa$-many elements, and suppose that such $\Age_\kappa$ has joint-embedding and amalgamation properties. Then there is an $L$-structure $U$ generated by at most $\lambda$-many elements which contains isomorphically as a substructure every $L$-structure from $\Age_\lambda$.

Indeed, this is just an iteration of the Fra\" iss\' e construction. Start with some uncountable cardinal $\kappa<\lambda$. Since $(\kappa^+)^\kappa=2^\kappa=\kappa^+$, we have that $\Age_\kappa$ has a Fra\" iss\' e limit $U_\kappa$. Since $\Age_\lambda$ is closed under taking direct limits, $U_\kappa\in\Age_\lambda$. By induction, we can repeat the construction for every uncountable cardinal $\kappa<\lambda$ to obtain $U_\kappa$ in such a way that for $\kappa_1<\kappa_2<\lambda$, $U_{\kappa_1}\subseteq U_{\kappa_2}$. The direct limit $U$ of the system $\{U_\kappa:\kappa<\lambda\}$ is then the desired universal $L$-structure. We note that although $U$ is universal, it does not necessarily have the homogeneity property of Fra\" iss\' e limits.
\end{observation}
\begin{thm}\label{GCH_univ}
Let $\kappa$ be an uncountable cardinal such that $\kappa^{<\kappa}=\kappa$. Then there exists a metrically universal group $\mathbb{G}_\kappa$, with bi-invariant metric, of density/weight $\kappa$. Thus any group equipped with bi-invariant metric which has density less or equal to $\kappa$ embeds into $\mathbb{G}_\kappa$ via an isometric monomorphism.
\end{thm}
Note that now there is no restriction on the diameter of $\mathbb{G}_\kappa$. $\mathbb{G}_\kappa$ is unbounded and contains isometric copies of unbounded metric groups.\\

The construction is analogous to the construction of the universal $\mathbb{G}_K$ from the first section, just much easier. We consider the class $\Age_\kappa$ of all free groups of strictly less than $\kappa$ generators equipped with bi-invariant {\it pseudometric}. The pseudometric might be arbitrary, it does not have to be rational and it does not have to be generated by values on some proper subset of the free group. The reason for that is that there are still only $\kappa$-many such groups. However, we restrict the class of morphisms between them to those that send generators to generators, as in the countable case.

The amalgamation is then proved analogously. Suppose we have $G_0$, $G_0\ast F_1$ and $G_0\ast F_2$ where the pseudometric $d_i$, $i\in \{1,2\}$, on $G_0\ast F_i$ when restricted to $G_0$ is equal to the pseudometric on $G_0$. Then we form the canonical algebraic amalgamation $G_0\ast F_1\ast F_2$. Let $d'\supseteq d_1\cup d_2$ be the partial pseudometric on $G_0\ast F_1\ast F_2$ defined so that for any $a\in G_0\ast F_1$, $b\in G_0\ast F_2$, where $G_0$, $G_0\ast F_1$ and $G_0\ast F_2$ are all considered as subgroups of $G_0\ast F_1\ast F_2$, we have $$d'(a,b)=\inf \{d_1(a,c)+d_2(c,b):c\in G_0\}.$$ We then extend $d'$ to the whole $G_0\ast F_1\ast F_2$ canonically; i.e. for any $a,b\in G_0\ast F_1\ast F_2$ we set $$d(a,b)=\inf\{d'(a_1,b_1)+\ldots+d'(a_n,b_n):$$ $$a=a_1\cdot\ldots\cdot a_n,b=b_1\cdot\ldots\cdot b_n,a_1,b_1,\ldots,a_n,b_n\in G_0\ast F_1\cup G_0\ast F_2\}.$$

We thus get a certain limit $G_\kappa$ of this amalgamation class which is a free group with $\kappa$-many generators equipped with bi-invariant pseudometric. We take the quotient and then the completion (actually, an easy argument shows that the limit was already complete). This will be the desired group $\mathbb{G}_\kappa$. Let $G$ be any group with bi-invariant metric and density at most $\kappa$. Using Theorem \ref{free-dense} as in the countable case we may suppose that $G$ has the free group of $\kappa$-many generators as a dense subgroup. Using universality of $G_\kappa$ we can find an isometric copy of this free group in $G_\kappa$. It follows that $G$ then lies in $\mathbb{G}_\kappa$.\\

Using Theorem \ref{GCH_univ} and ideas from Observation \ref{observ_limit} we can, under GCH, construct metrically universal group with bi-invariant metric even for limit uncountable cardinals, so those for which Theorem \ref{GCH_univ} cannot be applied directly.
\begin{cor}\label{corGCH}
Under GCH, there is a metrically universal group $\mathbb{G}_\kappa$, with bi-invariant metric, of density/weight $\kappa$, for an arbitrary uncountable cardinal $\kappa$.
\end{cor}
\begin{proof}
If $\kappa$ is a successor cardinal, then $\kappa^{<\kappa}=\kappa$, and so we use the Fra\" iss\' e construction above. So suppose now that $\kappa$ is a limit of infinite cardinals $(\lambda_\beta)_{\beta<\alpha}$, where $\alpha$ is the cofinality of $\kappa$ and for each $\beta<\alpha$, if $\beta$ is isolated, then $\lambda_\beta$ is isolated as well. We use the ideas already explained in Observation \ref{observ_limit}. For each isolated $\beta<\alpha$ there is a group $G_\beta$ which is the Fra\" iss\' e limit of the class $\Age_{\lambda_\beta}$ as above. Moreover, we can suppose that for each $\beta_1<\beta_2$ we have that $G_{\beta_1}$ is a subgroup of $G_{\beta_2}$ and for each $\gamma<\alpha$ limit we have a group $G_\gamma$ defined as the direct limit of $(G_\beta)_{\beta<\gamma}$. Then we define $G_\alpha$ to be the direct limit of $(G_\beta)_{\beta<\alpha}$ and we set $\mathbb{G}_\kappa$ to be the quotient and completion of $G_\alpha$. We claim that $\mathbb{G}_\kappa$ is metrically universal for 
groups with bi-invariant metric of density at most $\kappa$.

Let $\mathbb{H}$ be a metric group with bi-invariant metric with density at most $\kappa$. Suppose without loss of generality that the density is exactly $\kappa$. We denote by $H$ a dense subgroup of size $\kappa$ and write $H$ as the direct limit of subgroups $(H_\beta)_{\beta<\alpha}$, where $H_\beta$, for $\beta<\alpha$, has size $\beta$. Take some isolated $\beta_0<\alpha$. Then we can find an isometric monomorphism $\phi_{\beta_0}:H_{\beta_0}\rightarrow G_{\beta_0}$. By the extension property of $G_{\beta_0+1}$ we can extend it to an isometric monomorphism $\phi_{\beta_0+1}: H_{\beta_0+1}\rightarrow G_{\beta_0+1}$. We continue analogously by induction, taking direct limits at the limit steps. At the end we get an isometric monomorphism $\phi:H\rightarrow G_\alpha\leq \mathbb{G}_\kappa$. Then the completion $\mathbb{H}$ of $H$ lies inside $\mathbb{G}_\kappa$ isometrically.
\end{proof}
\begin{cor}
Under GCH, there is a universal SIN group of weight $\kappa$ for every infinite cardinal $\kappa$.
\end{cor}
\begin{proof}
Let $\kappa$ be an infinite cardinal. If $\kappa$ is countable then we are done by Theorem \ref{main2}, so suppose that $\kappa$ is uncountable. By Corollary \ref{corGCH}, there is a metrically universal group $\mathbb{G}_\kappa$ with bi-invariant metric of density $\kappa$. Denote by $\mathbb{H}$ the product of $\kappa$-many copies of $\mathbb{G}_\kappa$. It is a topological SIN group of weight $\kappa$. Now let $F$ be any topological SIN group of weight at most $\kappa$. It is well known that the topology of $F$ is determined by a family $(p_\alpha)_{\alpha<\kappa}$ of $\kappa$-many bi-invariant pseudometrics (we allow repetitions in the family, so we assume its size is indeed exactly $\kappa$). For each $\alpha<\kappa$ denote by $F_\alpha$ the metric group which is the quotient of $F$ with respect to the pseudometric $p_\alpha$; i.e. $F/N_\alpha$, where $N_\alpha$ is the closed normal subgroup $\{f\in F:p_\alpha(f,1)=0\}$. Clearly, $F$ embeds topologically as a subgroup of $\prod_{\alpha<\kappa} F_\alpha$,
 thus it suffices to embed $\prod_{\alpha<\kappa} F_\alpha$ into $\mathbb{H}$. However, that is immediate since $\mathbb{H}
=\prod_{\alpha<\kappa} \mathbb{G}_\kappa$, and for each $\alpha<\kappa$, $F_\alpha$ embeds isometrically into $\mathbb{G}_\kappa$.
\end{proof}
\subsection{Open problems}
Let us start with a problem related to a topic in geometric group theory. The following question was asked to us by Pestov.
\begin{question}\label{que1}
Can the universal group $\mathbb{G}$ be obtained as a subgroup of a metric ultraproduct of finite groups with bi-invariant metric?
\end{question}
We shall comment on this question more. Let $\mathcal{M}$ be some class of groups equipped with bi-invariant metric. Let $G$ be an arbitrary (discrete) group. We say that $G$ is \emph{$\mathcal{M}$-approximable} if there is some constant $K$ such that for every finite subset $F\subseteq G$ and every $\varepsilon>0$ there exist a group $(H,d)\in \mathcal{M}$ and a map $\varphi: F\rightarrow H$ such that
\begin{itemize}
\item for every $f,g\in F$ such that $f\cdot g\in F$ we have $d(\varphi(f)\cdot\varphi(g),\varphi(f\cdot g))<\varepsilon$,
\item if $1\in F$ then we have $d(1_H,\varphi(1))<\varepsilon$,
\item for every $f\neq g\in F$ we have $d(\varphi(f),\varphi(g))\geq K$.

\end{itemize}
An equivalent definition is the following: $G$ is $\mathcal{M}$-approximable if any finitely generated subgroup of $G$ (algebraically) embeds into a metric ultraproduct of groups from $\mathcal{M}$.

The most interesting cases are when $\mathcal{M}$ is the set of unitary groups of finite rank equipped with the Hilbert-Schmidt distance and when $\mathcal{M}$ is the set of finite permutation groups equipped with the normalized Hamming distance. For the former, such $\mathcal{M}$-approximable groups are then called hyperlinear, and for the latter, such $\mathcal{M}$-approximable groups are called sofic. The major open problem is whether every group is hyperlinear and sofic. We refer the reader to the survey \cite{Pe-sh} about these groups.

When $\mathcal{M}$ is the set of all finite groups equipped with some bi-invariant metric, then such $\mathcal{M}$-approximable groups are called weakly sofic. Weakly sofic groups as a generalization of sofic groups were introduced by Glebsky and Rivera in \cite{GlRi} as the existence of a non-weakly sofic group is equivalent to a certain conjecture about pro-finite topology on finitely generated free groups.

Positive answer to Question \ref{que1} would thus imply that every group is weakly-sofic. Indeed, it is enough to prove it for countable groups (resp. just finitely generated). Each such a group can be equipped with a trivial bi-invariant metric (which takes values just $0$ and $1$) and it thus exists as a subgroup of $\mathbb{G}$.

The most direct approach how to prove it would be to answer the following question which is in our opinion interesting in its own right. We will see it will also imply the next mentioned problem, so it is in our opinion the most interesting and important question posed in this paper.
\begin{question}\label{que2}
Let $(F_n,\rho)$ be a finitely generated free group with a bi-invariant metric and let $A\subseteq F_n$ be its finite subset and $\varepsilon>0$. Does there exist a finite group $F$ equipped with some bi-invariant metric and partial monomorphism $\iota: A\subseteq F_n\rightarrow F$ which is also an $\varepsilon$-isometry?
\end{question}
Notice that a positive answer to Question \ref{que2} would also lead to a positive answer to Question \ref{que1}. Indeed, note at first, that $G$, the canonical dense subgroup of $\mathbb{G}$ obtained by the Fra\" iss\' e construction, is a direct limit of certain sequence $(G_n)_n$ of finitely generated free groups with rational finitely generated metric. Assume that for each $n$, the metric on $G_n$ is generated by values on a finite subset $A_n\subseteq G_n$. Suppose that for each $n$ we have a finite group $F_n$ with bi-invariant metric and a partial monomorphism $\iota_n: A_n\subseteq G_n\rightarrow F_n$ which is an $1/2^n$-isometry. Then it is straightforward to check that the metric ultraproduct of the sequence $(F_n)_n$ would contain isometrically $G$, and thus, since it is complete, also $\mathbb{G}$.\\

A possible approach to Question \ref{que2} is to find a finite group $F$ with a finite subset $A'\subseteq F$ such that there is a partial monomorphism $\nu:A\subseteq F_n\rightarrow A'\subseteq F$. Such $F$ always exists since free groups are residually finite. Moreover, we can assume that $A'$ generates $F$. Then one can try to define a bi-invariant metric on $F$ as follows: for every $a,b\in F$ set $$d_F(a,b)=\min\{\rho(\nu^{-1}(a_1),\nu^{-1}(b_1))+\ldots+\rho(\nu^{-1}(a_n),\nu^{-1}(b_n)):$$ $$a=a_1\cdot\ldots\cdot a_n,b=b_1\cdot\ldots\cdot b_n,\, \forall i\leq n (a_i,b_i)\in (A')^2\}.$$ Note however that there is no guarantee we will have for every pair $a,b\in A'$ that $d_F(a,b)=\rho(\nu^{-1}(a),\nu^{-1}(b))$. Can $F$ be chosen in such a way that this holds true?\\

We were directed to the following question by Tsankov: 
\begin{question}\label{queTsankov}
Is $\mathbb{G}$ extremely amenable?
\end{question}
Recall that a group is extremely amenable if whenever it acts continuously on a compact Hausdorff space, then it has a fixed point (we refer to \cite{Pe} for more information about extreme amenability). In \cite{MeTs}, Melleray and Tsankov prove that for any `suitable' countable abelian group $G$, the set of those bi-invariant metrics $d$ on $G$ such that $(G,d)$ is extremely amenable is dense $G_\delta$ (see Theorem 6.4 in \cite{MeTs}). The natural question is to ask whether the same conclusion holds when $G$ is the free group of countable many generators. If the answer is positive, then together with Theorem \ref{genericthm} we will get that $\mathbb{G}$ is indeed extremely amenable.

We were informed by Tsankov that a sufficient requirement to get such a conclusion for $F_\infty$, and so to prove that $\mathbb{G}$ is extremely amenable, would be to prove that in the set of all pseudometrics on $F_\infty$, those pseudometrics whose quotients give finite groups form a dense subset. Curiously, one can readily check that this is an equivalent problem with that one from Question \ref{que2}. So not only does a positive solution of Question \ref{que2} imply a positive solution to Question \ref{que1} (and thus implies that all groups are weakly sofic), it also implies a positive answer to Question \ref{queTsankov}.\\

Our next problem is connected to the properties of the distance $\dist$. It was defined for free groups, resp. for free abelian groups in \cite{Do1}. Analogous definition is possible for any (finitely generated) group. Let $G$ be a finitely generated group and let $|.|:G\rightarrow \Nat$ be its Cayley distance, i.e. a graph metric in the Cayley graph of $G$ for some specified finite set of generators. Let $\mathcal{D}_G$ be the set of all bi-invariant metrics on $G$. For $\rho_1,\rho_2\in \mathcal{D}_G$ we set $$\dist_G(\rho_1,\rho_2)=\inf\{\varepsilon: \forall g\in G (|\rho_1(g,1)-\rho_2(g,1)|\leq \varepsilon\cdot |g|\}.$$ The main result from the previous section relies on the fact that if $G$ is a free group, then $(\mathcal{D}_G,\dist_G)$ is not separable. While in \cite{Do1}, it was proved, resp. it follows from results there, that if $G$ is a free abelian group, then $(\mathcal{D}_G,\dist_G)$ is separable.

We have the following question.
\begin{question}
For which finitely generated groups $G$ we have that $(\mathcal{D}_G,\dist_G)$ is separable?
\end{question}
Is the property of having such a space of metrics separable any related to amenability or (sub)exponential growth? It is apparently somehow connected to the geometry of $G$ in the sense of geometric group theory.\\

The last questions are all about some form of universality. We start with a question asked by Shkarin in \cite{Sk} where it appears as Problem 2.
\begin{question}
Does there exist a separable group with left-invariant metric that is metrically universal for the class of separable groups with left-invariant metric?
\end{question}
Of course, such a group could not be complete (and could not be completed) as then it would belong to a special class of groups that is known to not have a universal object by the result of Malicki (\cite{Ma}). Maybe one could produce a universal group with incomplete metric (whose completion would have to be only a semigroup).\\

We proved in Theorem \ref{nonuniv} that there is no metrically universal locally compact group with bi-invariant metric. However, we may ask for a bounded one.
\begin{question}
Does there exist a separable locally compact group $G$ with bi-invariant metric bounded by $1$ such that every separable locally compact group with bi-invariant metric bounded by $1$ embeds via an isometric isomorphism?
\end{question}
We have a similar question for compact groups. Here the requirement on boundedness of the metric is also obvious.
\begin{question}
Does there exist a compact group $G$ with bi-invariant metric bounded by $1$ such that every compact group with bi-invariant metric bounded by $1$ embeds via an isometric isomorphism?
\end{question}

\end{document}